\documentclass[12pt,twoside]{article}

\usepackage{a4wide, amsthm, amssymb, amscd, mathptmx}

\usepackage{amsfonts}

\usepackage[all]{xy}

\theoremstyle{plain}
\newtheorem{Thm}{\sc Theorem}[section]
\newtheorem{Theorem}[Thm]{\sc Theorem}
\newtheorem{Corollary}[Thm]{\sc Corollary}
\newtheorem*{Corollary*}{\sc Corollary}
\newtheorem{Proposition}[Thm]{\sc Proposition}
\newtheorem*{Proposition*}{\sc Proposition}
\newtheorem{Lemma}[Thm]{\sc Lemma}

\theoremstyle{definition}
\newtheorem{Definition}[Thm]{Definition}

\theoremstyle{remark}
\newtheorem{Remark}[Thm]{Remark}
\newtheorem{Example}[Thm]{Example}
\newtheorem*{Example*}{Example}
\newtheorem*{Remark*}{Remark}

\newcommand{\bB}{{\mathbf B}}

\newcommand{\A}{{\cal A}}
\newcommand{\B}{{\cal B}}
\newcommand{\C}{{\cal C}}
\newcommand{\D}{{\cal D}}
\newcommand{\E}{{\cal E}}
\newcommand{\F}{{\cal F}}
\newcommand{\G}{{\cal G}}
\renewcommand{\H}{{\cal H}}

\newcommand{\M}{{\cal M}}

\renewcommand{\O}{{\cal O}}

\newcommand{\T}{{\cal T}}
\newcommand{\X}{{\cal X}}
\newcommand{\Q}{{\cal Q}}
\newcommand{\fg}{{\mathfrak g}}

\renewcommand{\AA}{\mathbb{A}}
\newcommand{\BB}{\mathbb{B}}

\newcommand{\GG}{\mathbb{G}}
\newcommand{\HH}{\mathbb{H}}
\newcommand{\PP}{\mathbb{P}}
\newcommand{\QQ}{\mathbb{Q}}

\newcommand{\ZZ}{\mathbb{Z}}
\newcommand{\Gm}{\mathbb{G}_m}

\newcommand{\Ad}{\mathop{\rm Ad}}

\newcommand{\ch}{\mathop{\rm ch}}
\renewcommand{\char}{\mathop{\rm char}}
\newcommand{\Coh}{\mathop{{ Coh \, }}}
\newcommand{\coker}{\mathop{\rm coker}}
\newcommand{\enD}{\mathop{{\rm End}}}

\newcommand{\ext}{{\mathop{{\rm Ext}}}}

\newcommand{\GL}{\mathop{\rm GL}}
\newcommand{\Gr}{\mathop{\rm Gr}}
\newcommand{\gr}{\mathop{\rm gr}}
\newcommand{\Group}{\mathop{\rm Group}}

\newcommand{\Hom}{{\mathop{{\cal H}om}}}
\renewcommand{\hom}{{\mathop{\rm Hom}}}
\newcommand{\id}{\mathop{\rm Id}}
\newcommand{\Id}{\mathop{\rm Id}}
\newcommand{\im}{\mathop{\rm im}}

\newcommand{\Mat}{\mathop{\rm Mat}}
\newcommand{\Mod}{\mathop{\rm Mod}}

\newcommand{\Ob}{\mathop{\rm Ob}}
\newcommand{\Perv}{ {\mathop{\rm Perv}}}

\newcommand{\pic}{{\mathop{\rm Pic \,}}}

\newcommand{\Proj}{ {\mathop{\rm Proj}}}

\newcommand{\reg}{\mathop{\rm reg}}
\newcommand{\rk}{\mathop{\rm rk}}
\newcommand{\Sch}{\mathop{\rm Sch}}
\newcommand{\Sets}{\mathop{\rm Sets}}

\newcommand{\Spec}{\mathop{\rm Spec}}
\newcommand{\Supp}{\mathop{\rm Supp}}

\newcommand{\ti}{\tilde}

\newcommand{\Tr}{\mathop{\rm Tr}}

\def\fal#1{\tilde{#1}}
\def\gr#1{{\bf #1}}
\def\hox{H^0(\O_X(1))}
\def\izo{\simeq}
\newcommand{\ev}{\mathop{\rm ev}}
\def\lin{\mathop{\rm span}}

\begin{document}

\markboth{\rm M.\ Hauzer, A. Langer} {\rm Moduli spaces of
instantons}

\title{Moduli spaces of framed perverse  instantons on $\PP ^3$}
\author{Marcin Hauzer \footnote{This author tragically died on 27.01.2010.
}, Adrian Langer}
\date{\today}

%%%
\maketitle
%%%

{\sc Addresses of A.L.:}\\
1. Institute of Mathematics, Warsaw University, ul.\ Banacha 2,
02-097 Warszawa, Poland,\\ e-mail: {\tt
alan@mimuw.edu.pl}\\
2. Institute of Mathematics, Polish Academy of Sciences, ul.
\'Sniadeckich 8, 00--956 Warszawa, Poland

\begin{abstract}
We study moduli spaces of framed perverse instantons on $\PP ^3$.
As an open subset it contains the (set-theoretical) moduli space
of framed instantons studied by I. Frenkel and M. Jardim in
\cite{FJ}. We also construct a few counterexamples to earlier
conjectures and results concerning these moduli spaces.
\end{abstract}

\section*{Introduction}

A mathematical instanton  is a torsion free sheaf $E$ on $\PP ^3$
such that $H^1(E(-2))=H^2(E(-2))=0$ and there exists a line on
which $E$ is trivial. It is conjectured that the moduli space of
locally free instantons of rank $2$ is smooth and irreducible but
this is known only for very small values of the second Chern class
$c$ (see \cite{CTT} and \cite{KO} for proof of this conjecture for
$c\le 5$ and history of the problem).

Originally, instantons appeared in physics as anti-selfdual
connections on the $4$-dimensional sphere. Later, they were
connected by the ADHM construction to mathematical instantons on
$\PP ^3$ with some special properties. But it was Donaldson  who
realized that there is a bijection between physical instantons on
the $4$-sphere with framing at a point and vector bundles on a
plane framed along a line (see \cite{Do}). The correspondence can
be seen using Wards' construction and restricting vector bundles
from $\PP ^3$ to a fixed plane containing the line corresponding
to the point of the sphere. Using this interpretation Donaldson
was able to conclude that the moduli space of physical instantons
is smooth and irreducible.

In \cite{FJ} Frenkel and Jardim started to investigate the moduli
space of mathematical instantons framed along a line, hoping that
this moduli space is easier to handle than the moduli space of
instantons. Since an open subset of the moduli space  of framed
instantons is a principal bundle over the moduli space of
instantons, it is sufficient to consider the conjecture in the
framed case. In fact, Frenkel and Jardim conjectured that their
framed moduli space is smooth and irreducible even at non-locally
free framed instantons. We show that this conjecture is false (see
Subsection \ref{counter}). On the other hand, we also show that
the moduli space of locally free framed instantons is smooth for
low ranks and values of the second Chern class (see Corollary
\ref{smoothness-for-low-charge}). We also use Tyurin's idea to
show that in some case the restriction map embeds the moduli space
of instantons as a Lagrangian submanifold into the moduli space of
sheaves on a quartic in $\PP^3$.

One of the main aims of this paper is the study of moduli spaces
of perverse framed instantons on $\PP ^3$ (see Definition
\ref{perv-definition}). In particular, we use perverse instantons
to introduce partial compactifications of Gieseker and
Donaldson-Uhlenbeck type of the moduli space of framed instantons
and study the morphism between these moduli spaces. The picture
that we get is quite similar to the one known from the plane case
(see \cite{Na1}) or from the study of a similar morphism for
sheaves on surfaces (see, e.g., \cite[Remark 8.2.17]{HL2}).
However, this is the first case when a similar morphism is
described for moduli spaces of sheaves on a $3$-dimensional
variety.

In the $2$-dimensional case the Donaldson-Uhlenbeck
compactification has a stratification by products of moduli spaces
of locally free sheaves for smaller second Chern class and
symmetric powers of a plane $\AA^2$. In our case the situation is
quite similar but more complicated: we get a stratification by
products of moduli spaces of regular perverse instantons and
moduli spaces of perverse instantons of rank $0$.

Perverse instantons of rank $0$ are sheaves $E$ of pure dimension
$1$ on $\PP ^3$ such that $H^0(E(-2))=H^1(E(-2))=0$. The moduli
space of such sheaves (with fixed second Chern class) has a
similar type as a Chow variety: it is only set-theoretical and it
does not corepresent the moduli functor of such sheaves. But the
moduli space of perverse rank $0$ instantons is still a coarse
moduli space for some functor: it is the moduli space of modules
over some associative (but non-commutative) algebra. We show that
this moduli space contains an irreducible component whose
normalization is the symmetric power of $\AA^4$.

\medskip

The structure of the paper is as follows. In Section 1 we recall a
few known results including Nakajima's description of the moduli
space of framed torsion free sheaves on a plane and
Frenkel--Jardim's description of the (set-theoretical) moduli
space of framed instantons in terms of ADHM data. In Section 2 we
introduce perverse instantons and we sketch proof of
representability of the stack of framed perverse instantons on
$\PP ^3$ (in the plane case this theorem is due to Drinfeld; see
\cite{BFG}). Then in Section 3 we study the notion of stability of
ADHM data in terms of Geometric Invariant Theory. This is crucial
in Section 4 where we describe the Gieseker and
Donaldson--Uhlenbeck type compactifications of the moduli space of
framed instantons. In Section 5 we study the moduli space of
perverse instantons of rank $0$ relating them to the moduli space
of modules over a certain non-commutative algebra. In particular,
we show an example when this moduli space is reducible. In Section
6 we gather several examples and counterexamples to some
conjectures, e.g., to the Frenkel--Jardim conjecture on smoothness
and irreducibility of moduli space of torsion free framed
instantons or to their conjecture on weak instantons. In Section 7
we study an analogue of the hyper-K\"ahler structure on the moduli
space of perverse instantons and we relate our moduli spaces to
moduli spaces of framed modules of Huybrechts and Lehn. In Section
8 we give a very short sketch of deformation theory for stable
framed perverse instantons and we study smoothness of moduli
spaces of framed locally free instantons.

\section{Preliminaries}

In this section we introduce notation and collect a few known
results needed in later sections.

\subsection{Geometric Invariant Theory (GIT)} \label{GIT-section}

Let $G$ be a reductive group.  Let $X$ be an affine $k$-scheme
(possibly non-reduced or reducible) with a left $G$-action.  A
character $\chi : G\to \Gm$ gives a $G$-linearization of the
trivial line bundle $L:=(X\times \AA ^1\to X)$ via $g\cdot
(x,z)=(gx, \chi(g^{-1})z)$. So we can consider the corresponding
GIT quotient
$$X^{ss}(L)/G=\Proj (\bigoplus_{n\ge 0}H^0(X, L^n)^{G}).$$
It is equal to
$$X/\!\! /_{\chi} G =\Proj (\bigoplus_{n\ge 0}k[X]^{G, \chi^n})$$ and it is
projective over $X/G=\Spec (k[X]^G)$ (see \cite{Ki} for this
description). The corresponding map
$$X/\!\! /_{\chi} G \to X/G$$
can be identified with the map describing change of polarization
from $\chi$ to the trivial character $1: G\to \Gm$. The GIT
(semi)stable points of the $G$-action on $(X, L)$ given by $\chi$
are called \emph{$\chi$-(semi)stable}. Note that all points of $X$
are $1$-semistable, i.e., GIT semistable for the trivial character
$1$.

We say that $x$ is \emph{$\chi$-polystable} if $G\cdot (x,z)$ is
closed for $z\ne 0$. In particular, $X/\!\! /_{\chi} G$ is in
bijection with the set of $\chi$-polystable points and a
$\chi$-polystable point is $\chi$-stable if and only if its
stabiliser in $G$ is trivial.

\subsection{Torsion-free sheaves on $\PP ^2$ and ADHM data} \label{section:P^2}

Let $V$ and $W$ be $k$-vector spaces of dimensions $c$ and $r$,
respectively. Set
$$\gr{B}=\hom(V,V)\oplus\hom(V,V)\oplus\hom(W,V)\oplus\hom(V,W).$$
 An element of
${\gr{B}}$ is written as $({B_1},{B_2}, {i},{j})$.

The map $\mu: \gr{B}\to \hom (V,V)$ given by
$$\mu (B_1, B_2, i, j)=[B_{1}, B_{2}]+ i j$$
is called the \emph{moment map}.

We say that $( B_1,  B_2, i, j)\in \gr{B}$ satisfies the
\emph{ADHM equation} if $[ B_{1}, B_{2}]+ i j=0$, i.e., $( B_1,
B_2, i, j)\in \mu ^{-1}(0)$. An element of $ \gr{B}$ satisfying
the ADHM equation is called an \emph{ADHM datum}.

\begin{Definition}
We say that an ADHM datum is
\begin{enumerate}
\item \emph{stable}, if for every subspace $S\subsetneq V$ (note that we allow
$S=0$) such that $B_k(S)\subset S$ for $k=1,2$ we have $\im  i
\not\subset S$.
\item  \emph{costable},
if for every no non-zero subspace $S\subset V$ such that $
B_{k}(S)\subset S$ for $k=1,2$ we have $S\not \subset \ker  j$,
\item  \emph{regular}, if it is stable and costable.
\end{enumerate}
\end{Definition}

The group $G=\GL (V)$ acts on $\bB$ via
$$g\cdot (B_{1}, B_{2}, i, j)=(gB_{1}g^{-1}, gB_{2}g^{-1}, gi, jg^{-1}).$$
If we consider the adjoint action of $G$ on $\enD (V)$ then the
map $\mu$ is $G$-equivariant. In particular, $G$ acts on
$\ti\mu^{-1}(0)$, i.e., on the set of ADHM data satisfying the
ADHM equation. Let $\chi:G\to \Gm$ be the character given by the
determinant. We consider the $G$-action on the trivial line bundle
over $\gr{B}$ but with a non-trivial linearization given the
character $\chi$.

\begin{Lemma}

\begin{enumerate}
\item All $\chi$-semistable points of $\mu^{-1}(0)$ are
$\chi$-stable and they correspond to stable ADHM data.
\item All $\chi ^{-1}$-semistable points of $\mu^{-1}(0)$ are $\chi
^{-1}$-stable and they correspond to costable ADHM data.
\end{enumerate}
\end{Lemma}

We have the following well-known theorem (see \cite[Theorem 2.1,
Remark 2.2 and Lemma 3.25]{Na1}):

\begin{Theorem} \label{Donaldson-Nakajima}
The moduli space $\M (\PP ^2; r,c)$ of rank $r>0$ torsion free sheaves on
$\PP ^2$ with $c_2=c$, framed along a line $l_{\infty}$ is
isomorphic to the GIT quotient $\mu^{-1} (0)/\!\!/_{\chi} G$.
Moreover, orbits of regular ADHM data are in bijection with
locally free sheaves.
\end{Theorem}

\begin{Definition}
A complex of locally free sheaves
$$\C=(0\to \C^{-1}\mathop{\to}^\alpha \C^0\mathop{\to}^{\beta} \C^1\to 0)$$
is called a \emph{monad} if $\alpha$ is injective and $\beta$ is
surjective (as maps of sheaves). In this case
$\H^0(\C)=\ker\beta/\im \alpha$ is called the \emph{cohomology} of
the monad $\C$.
\end{Definition}

Now let us briefly recall how to recover a torsion free sheaf from
a stable ADHM datum.

Let $( B_1,  B_2, i, j)\in \gr{B}$ be a stable ADHM datum. Denote
$\fal{W}=V\oplus V\oplus W$ and fix homogeneous coordinates
$[x_0,x_1,x_2]$ on $\PP^2$. Let us define maps $\alpha:
V\otimes\O_{\PP^2}(-1)\to  \fal{W}\otimes\O_{\PP^2}$ and $\beta :
\fal{W}\otimes\O_{\PP^2}\to V\otimes\O_{\PP^2}(1)$ by
\begin{equation}\alpha=\left(\begin{array}{c}
{B}_1x_0-1\otimes x_1\\
{B}_2x_0-1\otimes x_2\\
{jx_0}
\end{array}
\right)
\end{equation}
and
\begin{equation}\beta=\left(\begin{array}{ccc}
-{B}_2x_0+1\otimes x_2&{B}_1x_0-1\otimes x_1& ix_0
\end{array}
\right) .
\end{equation}
Then $(B_1,  B_2, i, j)$ gives rise to the complex
$$
V\otimes\O_{\PP^2}(-1)\mathop{\longrightarrow}^{\alpha}\fal{W}\otimes\O_{\PP^2}
\mathop{\longrightarrow}^{\beta}V\otimes\O_{\PP^2}(1).$$
This complex is a monad.  Injectivity of $\alpha$ follows from
injectivity on the line $x_0=0$ and surjectivity of $\beta$
follows from stability of the ADHM datum (see \cite[Lemma
2.7]{Na1}). We can recover a torsion free sheaf as the cohomology
of this monad.

\medskip
Let $\M _0^{\reg} (\PP ^2;r,c)$ be the moduli space of rank $r$ locally
free sheaves on $\PP ^2$ with $c_2=c$, framed along a line
$l_{\infty}$. By Theorem \ref{Donaldson-Nakajima} $\M _0^{\reg}
(\PP ^2, r,c)$ is isomorphic to the quotient of regular ADHM data by the
group $G$.

Let $\M _0(\PP ^2;r,c)$ denotes the affine quotient $\mu ^{-1} (0)/G$.
This space contains the moduli space $\M _0^{\reg} (\PP ^2;r,c)$ and it
can be considered as a partial Donaldson--Uhlenbeck
compactification. We have a natural set-theoretical decomposition
$$\M _0(\PP ^2;r,c)=\bigsqcup _{0\le d \le c} \M _0^{\reg} (\PP ^2;r,c-d)\times S^d (\AA ^2),$$
where $\AA ^2$ is considered as the completion of $l_{\infty}$ in
$\PP ^2$. Then the morphism
$$\M (\PP ^2;r,c)\simeq \mu ^{-1} (0)/\!\!/_{\chi} G\to \mu^{-1} (0)/ G \simeq \M _0(\PP ^2;r,c)$$
coming from the GIT  (see Subsection \ref{GIT-section}) can be
identified with the map
$$(E, \Phi)\to ((E^{**}, \Phi), \Supp ({E^{**}/E}))$$
(see \cite[Exercise 3.53]{Na1} and \cite[Theorem 1]{VV}). This
morphism is an analogue of the morphism from the Gieseker
compactification of the moduli space of  (semistable) locally free
sheaves on a surface by means of torsion free sheaves to its
Donaldson--Uhlenbeck compactification. In a very special case of
rank one this corresponds to the morphism from the Hilbert space
to the Chow space (the so called \emph{Hilbert--Chow morphism}).

\bigskip

In the rest of this section, to agree with the standard notation
we need to assume that the characteristic of the base field is
zero (or it is sufficiently large).

Let us define a symplectic form $\omega$ on $\bB$ by
$$\omega ( (B_{1}, B_{2}, i, j), (B_{1}', B_{2}', i', j')):=\Tr
(B_1B_2'-B_2B_1'+ij'-i'j).$$ We will use the same notation for the
form induced on the tangent bundle $T\bB$.  One can easily check
that $\mu$ is a \emph{momentum map}, i.e.,
\begin{enumerate}
\item $\mu$ is $G$-equivariant, i.e., $\mu (g\cdot x)=\Ad ^*_{g^{-1}}\mu (x)$,
\item $\langle d\mu_x(v), \xi\rangle =\omega (\xi_x, v)$ for any $x\in
\bB$, $v\in T_x\bB$ and $\xi \in \fg$ ($\xi_x$ denotes the image
of $\xi$ under the tangent of the orbit map of $x$).
\end{enumerate}
In particular, \cite[Lemma 3.2]{KLS} implies that the moduli space
of semistable sheaves is smooth (there are many others proofs of
this fact: a sheaf-theoretic proof is trivial but we mention the
above proof since another argument using ADHM data given in
\cite[Lemma 3.2]{Va} seems a bit too complicated).

\subsection{Mathematical instantons on $\PP ^3$.}
\label{section:inst}

\begin{Definition}
A torsion free sheaf $E$ on $\PP ^3$ is called a
\emph{mathematical $(r,c)$-instanton}, if $E$ has rank $r$,
$c_2E=c$, $H^1(\PP ^3, E(-2))=H^2(\PP ^3, E(-2))=0$ and there
exists a line $l\subset \PP^3$ such that the restriction of $E$ to
${l}$ is isomorphic to the trivial sheaf $\O _l^r$.

Let us fix a line $l_{\infty}\subset \PP^3$. A choice of an
isomorphism $\Phi:E|_{l_{\infty}}\stackrel{\simeq}{\rightarrow}
\O_{l_{\infty}}^r$ is called a \emph{framing} of $E$ along
$l_{\infty}$. A pair $(E,\Phi)$ consisting of a mathematical
$(r,c)$-instanton $E$ and its framing $\Phi:
E|_{l_{\infty}}\stackrel{\simeq}{\rightarrow} \O_{l_{\infty}}^r$
is called a \emph{framed $(r,c)$-instanton}.
\end{Definition}

In the following we skip adjective ``mathematical'' and we will
refer to mathematical instantons simply as  instantons.

The following lemma is well known (for locally free sheaves see
\cite[Chapter II, 2.2]{OSS}):

\begin{Lemma} \label{stab-inst}
If a torsion free sheaf $E$ on $\PP ^n$ is trivial on one line
then it is slope semistable. In particular, an instanton on
$\PP^3$ is slope semistable.
\end{Lemma}

\begin{proof}
The sheaf $E$ is trivial on a line $m\subset \PP^n$ if and only if
$E|_m$ is torsion free and $H^1(E|_m(-1))=0$. Since these are open
conditions it follows that if $E$ is trivial on one line then it
is trivial on a general line. Now if $E'\subset E$ then for a
general line $m$ we have $E'|_m\subset E|_m\simeq \O_{\PP^1}^{\rk
E}$, so $\mu (E')=\deg (E'|_m)/\rk E'\le 0$.
\end{proof}

\medskip

If $E$ is a rank $2$ locally free sheaf with $c_1E=0$ on $\PP ^3$
then $H^1(\PP ^3, E(-2))$ and $H^2(\PP ^3, E(-2))$ are Serre dual
to each other. Let us also recall that if $k$ has characteristic
$0$ then for a rank $2$ locally free sheaf $E$ with $c_1E=0$
existence of a line $l$ such that the restriction of $E$ to $l$ is
trivial is equivalent to $H^0(E(-1))=0$ (this follows from the
Grauert--M\"ulich restriction theorem). This leads to a more
traditional definition of rank $2$ instantons as rank $2$ vector
bundles on $\PP ^3$ with vanishing $H^0(E(-1))$ and $H^1(E(-2))$
(see \cite[Chapter II, 4.4]{OSS}). Usually, one also adds
vanishing of $H^0(E)$ which in this case is equivalent to slope
stability (if $H^0(E)\ne 0$ then $E\simeq \O_{\PP^3}^2$, so this
cannot happen if $c\ge 1$).

Note that if $E$ is locally free of rank $\ge 3$ then vanishing of
$H^2(\PP ^3, E(-2))$ does not follow from the remaining conditions
(see Example \ref{coanda}).

We say that a locally free sheaf $E$ is \emph{symplectic}, if it
admits a non-degenerate symplectic form (or equivalently, an
isomorphism $\varphi : E\to E^*$ such that $\varphi ^* =-
\varphi$). It is easy to see that a non-trivial  symplectic sheaf
has an even rank. Obviously, for a symplectic locally free sheaf,
vanishing of $H^2(\PP ^3, E(-2))$ follows from vanishing of
$H^1(\PP ^3, E(-2))$ (by the Serre duality).

\medskip

The following fact was known for a very long time:

\begin{Theorem} \emph{(Barth, Atiyah \cite[Theorem 2.3]{At})}
Let $E$ be a symplectic $(r,c)$-instanton. Then $E$ is the
cohomology of a monad
$$0\to \O _{\PP ^3}(-1)^c\to \O _{\PP ^3}^{2c+r}\to\O _{\PP ^3}(1)^c\to 0.$$
\end{Theorem}

In the following we will need the following lemma (it should be
compared with \cite[Proposition 15]{FJ} dealing with the rank $1$
case).

\begin{Lemma} \label{existence}
There exist framed locally free $(r,c)$-instantons if and only if
either $r=1$ and $c=0$ or $r>1$ and $c$ is an arbitrary
non-negative integer. Moreover, if there exist framed locally free
$(r,c)$-instantons then there exist framed locally free
$(r,c)$-instantons $F$ such that $\ext ^2(F,F)=0$.
\end{Lemma}

\begin{proof}
The case $r=0$ is clearly not possible. Let us first assume that
$r=1$. Since the only line bundle with trivial determinant is
$E=\O_{\PP ^1}$ we see that $c_2(E)=0$. So to finish the proof it
is sufficient to  show existence of locally free
$(2,c)$-instantons. If $E$ is such an instanton then $F=E\oplus
\O^{r-2}_{\PP ^3}$ can be given a structure of framed locally free
$(r,c)$-instanton.

Existence of locally free $(2,c)$-instantons is well known. For
example, we can use Serre's construction (see \cite[Chapter I,
Theorem 5.1.1]{OSS}) to construct the so called \emph{t'Hooft
bundles}. More precisely, let $L_1, \dots ,L_c$ be a collection of
$c$ disjoint lines in $\PP ^3$ and let $Y$ denotes their sum (as a
subscheme of $\PP ^3$). Then by the above mentioned theorem there
exists a rank $2$ locally free sheaf $E$ that sits in a short
exact sequence
$$0\to \O_{\PP ^3}\to E\to J_Y\to 0.$$
One can easily see that $E$ is a $(2,c)$-instanton. Moreover, it
is easy to see that $\ext ^2(F,F)=0$ as $\ext ^2(E,E)=0$ and
$H^2(E)=H^2(E^*)=0$ (note that $E^*\simeq E$).
\end{proof}

\subsection{Generalized ADHM data after Frenkel--Jardim} \label{FJ-section}

Let $X$ be a smooth projective variety over an algebraically
closed field $k$ and let  $\O_X(1)$ be a fixed ample line bundle.
Let $V$ and $W$ be $k$-vector spaces of dimensions $c$ and $r$,
respectively. Set
$$\gr{B}=\hom(V,V)\oplus\hom(V,V)\oplus\hom(W,V)\oplus\hom(V,W)$$
and $\fal{\gr{B}}=\gr{B}\otimes H^0(\O_X(1))$. An element of
$\fal{\gr{B}}$ is written as $(\fal{B_1},\fal{B_2},
\fal{i},\fal{j})$, where $\ti B_1$ and $\ti B_2$ are treated as
maps $V\to V\otimes H^0(\O_X(1))$, $\ti i$ as a map $W\to V\otimes
H^0(\O_X(1))$ and $\ti j$ as a map $V\to W\otimes H^0(\O_X(1))$.

Let us define an analogue of the moment map
$$\ti\mu=\ti\mu_{W,V}: \ti\bB\to \enD(V) \otimes H^0(\O_X(2))$$
by the formula
$$\ti\mu(\ti B_1, \ti B_2, \ti i, \ti j)=[\ti B_1, \ti B_2]+\ti i \ti j.$$

As before an element of $\ti \mu ^{-1}(0)$ is called an {\it ADHM
datum} (or an \emph{ADHM $(r,c)$-datum for $X$} if we want to show
dependence on $r$, $c$ and $X$).

\medskip

If we fix a point $p\in X$ then for a $k$-vector space $U$ the
evaluation map $\ev_p:\hox\to\O_X(1)_p\izo k$ tensored with
identity on $U$ gives a map $U\otimes\hox\to U$ which we also
denote by $\ev_p$. For simplicity, we will use the notation
$\fal{B}_1(p)=\ev_p\fal{B}_1\in\hom(V,V)$, etc. For an ADHM datum
$x=(\fal{B}_1,\fal{B}_2,\fal{i},\fal{j})$, $x(p)$ denotes the
quadruple $(\fal{B}_1(p),\fal{B}_2(p),\fal{i} (p),\fal{j} (p))$.
Note that for maps to be well defined we need to fix an
isomorphism $\O_X(1)_p\izo k$ at each point $p\in X$. This does
not cause any problems as all the notions that we consider are
independent of these choices.

\begin{Definition}\label{def_stable_ADHM}
We say that an ADHM datum $x\in \ti{\gr{B}}$ is
\begin{enumerate}
\item \emph{FJ-stable} (\emph{FJ-costable}, \emph{FJ-regular}), if $x(p)$ is stable
(respectively: costable, regular) for all $p\in X$,
\item \emph{FJ-semistable}, if there exists  a point $p\in X$ such that $x(p)$ is
stable,
\item \emph{FJ-semiregular}, if it is FJ-stable and there exists a point $p \in X$
such that $x(p)$ is regular.
\end{enumerate}
\end{Definition}

Definition \ref{def_stable_ADHM} in case of $X=\PP ^1$ (not $\PP
^3$!) was introduced by I. Frenkel and M. Jardim in \cite{FJ}, but
we slightly change the notation and we call stability,
semistability, etc. introduced in \cite{FJ}, FJ-stability,
FJ-semistability, etc. The reason for this change will become
apparent in later sections. Namely, in \cite{J} Jardim generalized
this definition of (semi)stability of ADHM data to all projective
spaces and claimed in \cite[Proposition 4]{J} that his notion of
semistability is equivalent to GIT semistability of ADHM data. We
show that this assertion is false.

\medskip

Let us specialize to the case $X=\PP ^1$. Let  $[x_0,x_1,x_2,x_3]$
be homogeneous coordinates in $\PP ^3$ and let us embedd $X$ into
$\PP ^3$ by $[y_0,y_1]\to [y_0,y_1,0,0]$. Then $x_0$ and $x_1$ can
be considered as elements of $H^0(X, \O _X(1))$.

Let us set $\fal{W}=V\oplus V\oplus W$. Then any point $x=
(\fal{B_1},\fal{B_2}, \fal{i},\fal{j})\in
\fal{\gr{B}}=\gr{B}\otimes H^0(\O_{\PP ^1}(1))$ gives rise to the
following maps of sheaves on $\PP ^3$: map $\alpha:
V\otimes\O_{\PP^3}(-1) \to \fal{W}\otimes\O_{\PP^3}$ given by
\begin{equation}\alpha=\left(\begin{array}{c}
\fal{B}_1+1\otimes x_2\\
\fal{B}_2+1\otimes x_3\\
\fal{j}
\end{array}
\right)
\end{equation}
and map $ \beta: \fal{W}\otimes\O_{\PP^3}\to
V\otimes\O_{\PP^3}(1)$ given by
\begin{equation}\beta=\left(\begin{array}{ccc}
-\fal{B}_2-1\otimes x_3&\fal{B}_1+1\otimes x_2&\fal{i} \end{array}
\right).
\end{equation}
It follows from easy calculations that $\beta\alpha=0$ if and only
if $(\fal{B}_1,\fal{B}_2,\fal{i},\fal{j})\in \ti \mu^{-1}(0)$. So
if $x$ is an ADHM datum then we get the complex
\begin{equation}
\label{monada} \C ^{\bullet}_x=({ V\otimes\O_{\PP^3}(-1)
\mathop{\longrightarrow}^{\alpha}\fal{W}\otimes\O_{\PP^3}\mathop{\longrightarrow}^{\beta}V\otimes\O_{\PP^3}(1))}
\end{equation}
considered in degrees $-1,0,1$.

Let $l_{\infty}$ be the line in $\PP ^3$ given by $x_0=x_1=0$. It
is easy to see that after restricting to $l_{\infty}$ the
cohomology of the above complex of sheaves becomes the trivial
rank $r$ sheaf on $\PP ^1$. Moreover, we have the following lemma:

\begin{Lemma}
\label{FJmonada} \emph{(see \cite[Proposition 11]{FJ})} Let us fix
an ADHM datum $x \in \fal{\gr{B}}$. Then the corresponding complex
$\C_x^{\bullet}$ is a monad if and only if the ADHM datum $x$ is
FJ-stable.
\end{Lemma}

\begin{proof}
The map $\alpha$ is always injective as a map of sheaves, so we
only need to check when $\beta$ is surjective. This is exactly the
content of \cite[Proposition 11]{FJ}.
\end{proof}

\medskip

The main theorem of \cite{FJ} is existence of the following
set-theoretical bijections:

\begin{Theorem} \emph{(\cite[Main Theorem]{FJ})}
The above construction of monads from ADHM data on $\PP ^1$ gives
bijections between the following objects:
\begin{itemize}
\item FJ-stable ADHM data and framed torsion free instantons;
\item FJ-semiregular ADHM data and framed reflexive instantons;
\item FJ-regular ADHM data and framed locally free instantons.
\end{itemize}
\end{Theorem}

\section{Perverse instantons on $\PP ^3$.}

In this section we introduce perverse sheaves and perverse
instantons and we show that perverse instantons are perverse
sheaves (this fact is non-trivial!). We also sketch proof of an
analogue of Drinfeld's representability theorem in the
$3$-dimensional case.

\subsection{Tilting and torsion pairs on $\PP ^3$.}

\begin{Definition} Let $\A$ be an abelian category.
A \emph{torsion pair} in $\A$ is a pair $(\T, \F)$ of full
subcategories of $\A$ such that the following conditions are
satisfied:
\begin{enumerate}
\item for all objects $T\in \Ob \T$ and $F\in \F$ we have $\hom _{\A}(T,
F)=0$,
\item for every object $E\in \Ob \A$ there exist objects $T\in \Ob \T$ and $F\in \Ob \F$
such that the following short exact sequence is exact in $\A$:
$$0\to T\to E\to F\to 0.$$
\end{enumerate}
\end{Definition}

We will need the following theorem of Happel, Reiten and Smal{\o}:

\begin{Theorem} \label{HRS} \emph{(see \cite[Proposition I.2.1]{HRS})}
Assume that $\A$ is the heart of a bounded $t$-structure on a
triangulated category $\D$ and suppose that $(\T, \F)$ is a
torsion pair in $\A$. Then the full subcategory
$$\B=\{E\in \Ob \D : H^i(E)=0\, \, {\rm for }\, \,   i\ne 0,-1, H^{-1}(E)\in \Ob \F, \, \, {\rm and }\, \,  H^0(E)\in \Ob \T \}$$
of $\D$ is the heart of a bounded $t$-structure on $\D$.
\end{Theorem}

\medskip
In the situation of the above theorem we say that $\B$ is obtained
from $\A$ by \emph{tilting} with respect to the torsion pair $(\T,
\F )$.

\medskip

Let $E$ be a coherent sheaf on a noetherian scheme $X$. The
\emph{dimension} $\dim E$ of the sheaf $E$ is  by definition the
dimension of the support of $E$. For a $d$-dimensional sheaf $E$
there exists a unique filtration
$$0\subset T_0(E)\subset T_1(E)\subset \dots \subset T_d(E)$$
such that $T_i(E)$ is the maximal subsheaf of $E$ of dimension
$\le i$ (see \cite[Definition 1.1.4]{HL2}).

\medskip

Let $\A$ be an abelian category. Then any object in $\A$ can be
viewed as a complex concentrated in degree zero. This yields an
equivalence between $\A$ and  the full subcategory of the derived
category $D(\A)$ of $\A$ of complexes $K^{\bullet}$ with
$H^i(K^{\bullet})=0$ for $i\ne 0$.

In the following by $D^b(X)$ we denote the bounded derived
category of the abelian category of coherent sheaves on the scheme
$X$. The object of $D^b(X)$ corresponding to a coherent sheaf $F$
is called a \emph{sheaf object} and by abuse of notation it is
also denoted by $\F$.

If $\C$ is a complex of coherent sheaves on $X$ then $\H^p(\C)$
denotes its $p$-th cohomology. We use this notation since we would
like to distinguish cohomology $\H^p(\F)$ of a sheaf object $\F$
and cohomology of a sheaf $H^p(\F)=H^p(X,\F)$.

\medskip

Let $\A=\Coh X$ be the category of coherent sheaves on a smooth
projective $3$-fold $X$. Let $\T$ be the full subcategory of $\A$
whose objects are all coherent sheaves of dimension $\le 1$. Let
$\F$ be the full subcategory of $\A$ whose objects are all
coherent sheaves $E$ which do not contain subsheaves of dimension
$\le 1$ (i.e., $T_1(E)=0$). Clearly, $(\T, \F)$ form a torsion
pair in $\A$.

\begin{Definition} \label{perv}
A complex $\C \in D^b(X)$ is called a \emph{perverse sheaf} if the
following conditions are satisfied:
\begin{enumerate}
\item $\H ^i (\C)=0$ for $i\ne 0,1$,
\item $\H ^0(\C)\in \Ob \F$,
\item $\H^1 (\C) \in \Ob \T$.
\end{enumerate}
\end{Definition}

\begin{Definition}
A \emph{moduli lax functor of perverse sheaves} is the lax functor
$\Perv (X) :\Sch/k \to \Group$ from the category of $k$-schemes to
the category of groupoids, which to a $k$-scheme $S$ assigns the
groupoid that has $S$-families of perverse sheaves on $X$ as
objects and isomorphisms of perverse sheaves as morphisms.
\end{Definition}

Theorem \ref{HRS} implies that perverse sheaves form an abelian
category which is a shift of the tilting of $\Coh X$ with respect
to the pair $(\T , \F)$. This fact is crucial in proof of the
following theorem (cf. \cite[VIII 5.1, 1.1, 1.2]{SGA1},
\cite[Theorem 3.5]{So}, \cite[Lemma 5.5]{BFG}):

\begin{Theorem}
The moduli lax functor of perverse sheaves on a smooth
$3$-dimensional projective variety is a (non-algebraic!)
$k$-stack.
\end{Theorem}

\begin{Proposition}
Let
$$\C=(\C^{-1} \mathop\to^{\alpha} \C ^0 \mathop\to^{\beta} \C^1 )$$
be a complex of locally free sheaves on $X$ and assume that there
exists a curve $j:C\hookrightarrow X$ such that $Lj^*\C$ is a
sheaf object in $D^b(C)$. Then the object of $D^b(X)$
corresponding to $\C$ is a perverse sheaf.
\end{Proposition}

\begin{proof}
Note that $Lj^*\C$ is represented by the complex
$$j^*\C^{-1} \mathop{\longrightarrow}^{j^*\alpha} j^*\C ^0 \mathop{\longrightarrow}^{j^*\beta} j^*\C^1.$$
Since $\H ^{-1}(Lj^*\C)=0$, the restriction of $\alpha $ to $C$ is
injective. Since $\C^{-1}$ is torsion free this implies that
$\alpha$ is injective and hence $\H ^{-1} (\C)=0$.

Similarly, by assumption we have $\H ^{1}(Lj^*\C)=0$ and hence the
restriction of $\beta $ to $C$ is surjective. This implies that
$\beta$ is surjective in codimension $1$ (i.e., it is surjective
outside of a subset of codimension $\ge 2$). Therefore
$\H^1(\C)=\coker \beta$ is a sheaf of dimension $\le 1$, i.e.,
$\H^1 (\C)$ is an object of $\T$.

By definition we have a short exact sequence
$$0\to \H ^0(\C) \to E=\coker \alpha \to \im \beta \to 0.$$
Therefore to finish the proof it is sufficient to show that
$T_1(E)=0$. To prove this let us consider the following
commutative diagram
$$
\begin{CD}
& &  & & 0   & &  0 \\
& &  & & @AAA   @AAA\\
0 @>>> \ker \gamma @>{\alpha}>> \C^0 @>{\gamma}>>  E/T_1(E) @>>> 0\\
& &  @AAA @AAA @AAA \\
0 @>>> \C^{-1} @>{\alpha}>> \C^0 @>>> E @>>> 0\\
& &  @AAA @AAA @AAA \\
& &  0 @>>> 0 @>>> T_1(E) \\
& & &   & & & @AAA \\
& & & &   & & 0\\
\end{CD}
$$ Using the snake lemma we get the following short
exact sequence
$$0\to \C^{-1}\to \ker \gamma \to T_1(E)\to 0.$$
Since $\ker \gamma$ is torsion free (as a subsheaf of $\C^0$), $\C
^{-1} $ is reflexive  and the map $\C^{-1}\to \ker \gamma$ is an
isomorphism outside of the support of $T_1 (E)$ (i.e., outside of
a subset of codimension $\ge 2$), the map $\C^ {-1}\to \ker
\gamma$ must be an isomorphism. In particular, $T_1 (E)=0$ and $\H
^0(\C)$ is an object of $\F$.
\end{proof}

\begin{Proposition} \label{perv2}
Let
$$\C=(\C^{-1} \mathop\to^{\alpha} \C ^0 \mathop\to^{\beta} \C^1 )$$
be a complex of locally free sheaves on $\PP ^3$ and assume that
there exists a curve $j:C\hookrightarrow \PP ^3$ such that
$Lj^*\C$ is a locally free sheaf object in $D^b(C)$. Then $\H^0
(\C)$ is torsion free.
\end{Proposition}

\begin{proof}
Let $Z$ be the set of points $p\in \PP ^3$ such that $\alpha
(p)=\alpha\otimes k(p): \C^{-1}\otimes k (p)\to \C^{0}\otimes k
(p)$ is not injective. It is a closed subset of $\PP ^3$ (in the
Zariski topology). Since $\H ^0 (Lj^*\C)$ is locally free, it is
easy to see that $Z$ does not intersect $C$ (if $Z\cap C\ne
\emptyset$ then the cokernel of $j^*\alpha$ would contain torsion
that would also be contained in $\H ^0 (Lj^*\C)$). But since we
are on $\PP ^3$ this implies that $Z$ has dimension at most $1$.
But the support of $T_2(\coker \alpha)$ is contained in $Z$ and
$T_2 (\coker \alpha)$ is  pure of dimension $2$ (by the previous
proposition). Therefore $\coker \alpha$ is torsion free, which
implies that $\H^0 (\C)$ is also torsion free.
\end{proof}

\subsection{Definition and basic properties of perverse instantons}

Let us denote by $j$ the embedding of a line $l$ into $\PP^3$. The
pull back $j^*$ induces the left derived functor $Lj^*: D^b(\PP
^3)\to D^b(l)$.

\begin{Definition} \label{perv-definition}
A rank $r$ \emph{perverse instanton} is an object $\C$ of the
derived category $D^{b}(\PP^3)$ satisfying the following
conditions:
\begin{enumerate}
\item  $H^p(\PP^3, \C \otimes \O_{\PP^3}(q))=0$ if either $p=0,1$ and $p+q<0$
or $p=2,3$ and $p+q\ge 0$,
\item $\H^p(\C )=0$ for $p\ne 0, 1,$
\item there exists a line $j:l  \hookrightarrow \PP ^3$ such that $Lj^*\C$ is isomorphic to
the sheaf object $\O _{l}^{\oplus r}$.
\end{enumerate}

Let us fix a line $j:l_{\infty}\hookrightarrow \PP ^3$ and choose
coordinates $[x_0,x_1,x_2,x_3]$ in $\PP ^3$ so that $l_{\infty}$
is given by $x_0=x_1=0$. A \emph{framing} $\Phi$ along
$l_{\infty}$ of a perverse instanton $\C$ is an isomorphism $\Phi:
Lj^*\C\to \O _{l_{\infty}}^{\oplus r}$. A \emph{framed perverse
instanton} is a pair $(\C, \Phi)$ consisting of a perverse
instanton $\C$ and its framing $\Phi$.
\end{Definition}

Any instanton is a perverse instanton. By the Riemann--Roch
theorem for any perverse instanton $\C$ there exists $c\ge 0$
such that $\ch (\C )=r-c [H]^2$. We also have a distinguished
triangle $\H^0(\C)\to \C\to \H^1(\C)[-1]\to \H^0[1]$.
However, it is not a priori  clear if a perverse instanton is a
perverse sheaf in the sense of Definition \ref{perv}. We will
prove that this is indeed the case in Corollary \ref{sense}.

\medskip

As before there is a natural $G=\GL (r)$-action on $\ti \gr{B}$
which induces a $G$-action on the set of ADHM data. More
precisely, let us recall that the group $G=\GL (V)$ acts on $\bB$
via
$$g\cdot (B_{1}, B_{2}, i, j)=(gB_{1}g^{-1}, gB_{2}g^{-1}, gi, jg^{-1})$$
and it induces the action on $\ti \bB$. If we consider the adjoint
action of $G$ on $\enD (V)$ then the map $\ti \mu$ is
$G$-equivariant. In particular, $G$ acts on $\ti\mu^{-1}(0)$,
i.e., on the set of ADHM data.

\medskip

The main motivation for introducing perverse instantons is the
following theorem:

\begin{Theorem} \label{bijection-ADHM-perverse}
There exists a bijection between isomorphism classes of perverse
instantons $(\C, \Phi)$  with $\ch (\C)=r-c[H]^2$ framed along
$l_{\infty}$ and $\GL (c)$-orbits of ADHM $(r,c)$-data for $\PP
^1$.
\end{Theorem}

The bijection in the theorem is the same as in Section
\ref{FJ-section}. Namely, if $x= (\fal{B_1},\fal{B_2},
\fal{i},\fal{j})$ is an $(r,c)$-complex ADHM datum then we can
associate to it  the complex
\begin{equation}
\C ^{\bullet}_x=({ V\otimes\O_{\PP^3}(-1)
\mathop{\longrightarrow}^{\alpha}\fal{W}\otimes\O_{\PP^3}\mathop{\longrightarrow}^{\beta}V\otimes\O_{\PP^3}(1))}
\end{equation}
where $\alpha$ and $\beta$ are defined as in Section
\ref{FJ-section}. This complex is a perverse instanton and it
comes with an obvious framing along $l_{\infty}$.

\medskip

In the following we sketch proof of a stronger version of the
above theorem showing that isomorphism already holds at the level
of stacks. To do so, first we need to generalize the above
definition to families of perverse instantons.

Let $S$ be a (locally noetherian) $k$-scheme. We set $j_S=j\times
{\id _S}: S\times _k {l_{\infty}} \to \PP^3_S=S\times _k\PP^3$.

\begin{Definition}An \emph{$S$-family of
framed perverse $(r,c)$-instantons} is an object $\C \in \Ob
D^b(\PP^3_ S)$ together with an isomorphism $\Phi : Lj_S^*\C\to
\O_{\l_{\infty}\times S}^r$ such that for every geometric point
$s: \Spec K \to S$, the derived pull-back $(Lj^*_s\C, Ls^*\Phi)$
is a framed perverse $(r,c)$-instanton on $\PP^3_K$.

A morphism $\varphi: (\C _1, \Phi _1)\to (\C _2, \Phi _2)$ of
$S$-families of framed perverse instantons is a morphism $\varphi
: \C_1 \to \C_2$ in $D^b (\PP^3_ S)$ such that $\Phi_1= \Phi
_2\circ Lj_S^*\varphi .$
\end{Definition}

\begin{Definition}
A \emph{moduli lax functor of framed perverse $(r,c)$-instantons}
is the lax functor $\Perv ^c_r(\PP^3, l_{\infty}) :\Sch/k \to
\Group$ from the category of $k$-schemes to the category of
groupoids, which to a $k$-scheme $S$ assigns the groupoid that has
$S$-families of framed perverse $(r,c)$-instantons as objects and
isomorphisms of framed perverse instantons as morphisms.
\end{Definition}

In order for the definition to make geometric sense we have to
note that the moduli lax functor is a {\sl stack}, i.e., it
defines a sheaf of categories in the faithfully flat topology:

\begin{Lemma}
The moduli lax functor of framed perverse $(r,c)$-instantons on
$\PP ^3$ is a $k$-stack of finite type.
\end{Lemma}

\medskip

Let $G$ be an algebraic group acting on a scheme $X$. Then we can
form a \emph{quotient stack} $[X/G]$ which to any scheme $S$
assigns the groupoid whose objects are pairs $(P,\varphi)$
consisting of a principal $G$-bundle $P$ on $S$ and a
$G$-equivariant morphism $\varphi : P\to X$. A morphism  in this
groupoid is an isomorphism $h: (P_1,\varphi _1)\to (P_2, \varphi
_2)$ of pairs, i.e., such an isomorphism $h:P_1\to P_2$ of
principal $G$-bundles that $\varphi _1=\varphi _2\circ h$.

\medskip

In the $3$-dimensional case we have the following analogue of
Drinfeld's theorem on representability of the stack of framed
perverse sheaves on $\PP^2$ (see \cite[Theorem 5.7]{BFG}):

\begin{Theorem}\label{perv-theorem}
The moduli stack $\Perv ^c_r(\PP^3, l_{\infty})$ is isomorphic to
the quotient stack $[{\ti\mu}^{-1}(0)/\GL (V)]$.
\end{Theorem}

Proof of the above theorem is analogous to proof of \cite[Theorem
5.7]{BFG} and it follows from the following two lemmas.

\begin{Lemma}
Let $\C$ be a perverse instanton on $\PP ^3$. Then
$$H^q(\PP ^3, \C (-1) \otimes \Omega _{\PP ^3}^{-p} (-p))=0$$
for $q\ne 1$ and for $q=1$, $p\le -3$ or $p>0$.
\end{Lemma}

This lemma and its proof are analogous to \cite[Lemma 2.4]{Na1}
and \cite[Proposition 26]{FJ}.

\begin{Lemma} \label{perv-family}
An $S$-family of perverse instantons $\C\in D^b(\PP^3_ S)$ is
canonically isomorphic to the complex of sheaves
$$\O _{\PP ^3_S}(-1) \otimes R^1(p_1)_* (\C \otimes  \Omega
^2(1))\mathop\to  ^{\alpha}\O _{\PP ^3_S} \otimes R^1(p_1)_* (\C
\otimes \Omega ^1) \mathop\to^{\beta} \O _{\PP ^3_S} (1) \otimes
R^1(p_1)_* (\C (-1))
$$
in degrees $-1,0,1$ coming from Beilinson's construction.
Moreover, $\alpha$ is injective (as a map of sheaves), the sheaves
$R^1(p_1)_*(\C (-1)\otimes \Omega ^p(p))$ are locally free for
$p=0,1,2$ and we have a canonical isomorphism
$$ R^1(p_1)_* (\C \otimes \Omega
^2(1)) \simeq R^1(p_1)_* (\C (-1)).$$
\end{Lemma}

The above lemma follows from the previous lemma by standard
arguments using Beilinson's construction (i.e., proof of existence
of Beilinson's spectral sequence) in families.

\begin{Corollary} \label{sense}
Let $\C$ be a  perverse instanton on $\PP ^3$. Then $\H^0(\C)$ is
torsion free and $\H^1 (\C)$ is of dimension $\le 1$.
\end{Corollary}

\begin{proof}
The assertion follows immediately from Proposition \ref{perv2} and
Lemma \ref{perv-family} applied for $S$ being a point.
\end{proof}

\medskip

\subsection{Analysis of singularities of perverse instantons}

\begin{Definition}
Let $E$ be a coherent sheaf on a smooth variety $X$. Then the set
of points where the sheaf $E$ is not locally free is called the
\emph{singular locus} of $E$ and it is denoted by $S(E)$.
\end{Definition}

It is easy to see that the singular locus of an arbitrary coherent
sheaf on $X$ is a closed subset of $X$ (in the Zariski topology).
Here we study the singular locus of perverse instantons on $\PP
^3$.

\medskip

From the proof of \cite[Proposition 10]{FJ} it follows that in
case of complex ADHM data if $\coker \alpha$ is not reflexive then
it is non-locally free along a certain (possibly non-reduced or
reducible) curve of degree $c^2$ (not $2c$!) that does not
intersect $l_{\infty}$. If $\coker \alpha$ is reflexive then it is
non-locally free only in a finite number of points.

We have two short exact  sequences:
$$0\to \H ^0(\C) \to \coker \alpha \to \im \beta \to 0$$
and
$$0\to \im \beta \to \C^1 \to \H^1(\C)=\coker \beta \to 0.$$
It follows that $\im \beta$ is torsion free and it is non-locally
free exactly along the support of $\H^1 (\C)$ (which is at most
$1$-dimensional).

Obviously, $\H^ 0(\C)$ can be non-locally free only at the points
of  $S(\coker \alpha)$ or at the points  of $S(\im \beta)$.
Moreover, the $1$-dimensional components of $S(\H ^0(\C))$ are
contained in $S(\coker \alpha)$. This follows from the fact that
the kernel of a map from a locally free sheaf to a torsion free
sheaf is reflexive.

\section{GIT approach to perverse instantons\label{ADHM}}

In this section we consider ADHM data for an arbitrary manifold
$X$. We have a natural $G=\GL (V)$-action on $\ti \gr{B}$ which
induces a $G$-action on the set of ADHM data $\ti\mu^{-1}(0)$. Let
$\chi:G\to \Gm$ be the character given by the determinant. We can
consider the $G$-action on $\fal{\gr{B}}\times \AA ^1$ with
respect to this character (i.e., a non-trivial $G$-linearization
of the trivial line bundle on $\ti {\gr{B}}$).

The main aim of this section is to study different notions of
stability obtained via Geometric Invariant Theory when taking
quotients $\ti\mu^{-1}(0)/\!\!/_{\chi} G$ and $\ti\mu^{-1}(0)/G$.

This section is just a careful rewriting of \cite[Section 2]{VV}
but we give a bit more details for the convenience of the reader.

\medskip

\begin{Definition}
We say that an ADHM datum is
\begin{enumerate}
\item \emph{stable}, if for every subspace $S\subsetneq V$ (note that we allow
$S=0$) such that $\ti B_k(S)\subset S\otimes H^0(\O_X(1))$ for
$k=1,2$ we have $\im \ti i \not\subset S\otimes H^0(\O_X(1))$.
\item  \emph{costable},
if for every no non-zero subspace $S\subset V$ such that $\ti
B_{k}(S)\subset S\otimes H^0(\O_X(1))$ for $k=1,2$ we have $S\not
\subset \ker \ti j$,
\item  \emph{regular}, if it is stable and costable.
\end{enumerate}
We say that $(\ti B_1, \ti B_2,\ti i,\ti j)$ satisfies the
\emph{ADHM equation} if $[\ti B_{1},\ti B_{2}]+\ti i \ti j=0.$
\end{Definition}

\medskip

The following lemma generalizes   \cite[Lemma 3.25]{Na1}. Its
proof is similar to the proof given in \cite{Na1}.

\begin{Lemma}
\label{lemat_z_nakajimy}
Let $x$ be an ADHM datum. Then $x$ is
stable if and only if $G\cdot(x,z)$ is closed for some (or,
equivalently, all) $z\neq 0$.
\end{Lemma}

\begin{proof}
Let sections $\{s_l\}$ form a basis of $\hox$. Then $\fal{B_k}$
and $\fal{i}$ can be written as
$$\fal{B_k}=\sum B_{lk}\otimes s_l,\qquad \fal{i}=\sum i_{l}\otimes s_l.$$
Assume that $G\cdot(x,z)$ is closed for $z\neq 0$. Suppose that
there exists $S\subsetneq V$ such that $\fal{B_k}(S)\subset
S\otimes\hox$ for $k=1,2$ and $\im\fal{i}\subset S\otimes\hox$.
Let us fix a subspace $S^{\perp}\subset V$ such that $V=S\oplus
S^{\perp}$. Then we have
$$B_{kl}=\left(\begin{array}{cc}
*&*\\
0&*
\end{array}\right),\qquad i_l=\left(\begin{array}{c}
*\\
0
\end{array}\right).$$
If we set $g(t)=\left(\begin{array}{cc}
1&0\\
0&t^{-1}
\end{array}\right)$ then we have
$$g(t)B_{kl}g(t^{-1})=\left(\begin{array}{cc}
*&t*\\
0&*
\end{array}\right),\qquad g(t)i_l=i_l.$$
Therefore there exists limit $\lim_{t\to 0}g(t)x$ in
$\fal{\gr{B}}$. On the other hand, when $t\to 0$ then
$$g(t)(x,z)=(g(t)x,\det(g(t))^{-1}z)=(g(t)x,t^{\dim S^{\perp}}z)$$
has a limit $(\lim_{t\to 0}g(t)x,0)$ which does not belong to
$G\cdot (x,z)$. Contradiction shows that $x$ has to be stable.

Now suppose that $x$ is stable and $G\cdot (x,z)$ is not closed.
By the Hilbert--Mumford criterion there exists a $1$-parameter
subgroup $\lambda:\Gm\to G$ such that the limit $\lim_{t\to
0}\lambda(t)(x,z)$ exists and it belongs to
$\overline{G\cdot(x,z)}\backslash G\cdot(x,z)$. Let $V(m)$ consist
of vectors $v\in V$ such that $\lambda(t)\cdot v=t^mv$ for every
$t\in \Gm$. Then we have a decomposition $V=\bigoplus_m V(m)$ and
we can choose a basis of $V$ such that
$$\lambda(t)=\left(\begin{array}{ccc}
t^{a_1}&&\\
&\ddots\\
&&t^{a_c}
\end{array}
\right)$$ where $a_1\geq \ldots\geq a_c$. Existence of
$\lim_{t\to 0}\lambda(t)\fal{B_k}\lambda(t^{-1})$ implies that the
limits $\lim_{t\to 0}\lambda(t)\fal{B_{lk}}\lambda(t^{-1})$ exist
for every $l$. Let $B_{lk}=(b_{ij})$. Then
$(\lambda(t)\fal{B_{lk}}\lambda(t^{-1}))_{ij}=t^{a_i-a_j}b_{ij}$.
This shows that $b_{ij}=0$ if $a_i<a_j$. Therefore
$\fal{B_k}(V(m))\subset (\bigoplus_{l\geq m}V(l))\otimes\hox$.
Similarly, one can show that $\im\fal{i}\subset (\bigoplus_{m\geq
0}V(m))\otimes\hox$. Let us set $S=\bigoplus_{m\geq 0}V(m)$. Then
$\fal{B_k}(S)\subset S\otimes\hox$ and $\im\fal{i}\subset
S\otimes\hox$, so from the stability condition it follows that
$S=V$. Therefore $\det\lambda(t)=t^N$ for $N\geq 0$. If $N=0$ then
$V(0)=V$ and $\lambda\equiv \id$. This is impossible because
$\lim_{t\to 0}\lambda(t)(x,z)\notin G\cdot(x,z)$. If $N>0$ then
$\lambda(t)(x,z)=(\lambda(t)x,\det(\lambda(t))^{-1}z)=(\lambda(t)x,t^{-N}z)$
which diverges as $t\to 0$. This gives a contradiction.
\end{proof}

\begin{Proposition}
\label{rownowaznosc_stabilnosci} The following conditions are
equivalent:
\begin{enumerate}
\item $x$ is stable,
\item $x$ is $\chi$-stable,
\item $x$ is $\chi$-semistable.
\end{enumerate}
Similar assertion holds if we replace stable with costable and
$\chi$ with $\chi^{-1}$.
\end{Proposition}
\begin{proof}
By Lemma \ref{lemat_z_nakajimy} $x$ is stable if and only if $x$
is $\chi$-polystable. So to prove the proposition it is sufficient
to prove that if $x$ is stable then its stabilizer in $G$ is
trivial. Assume that $g\in G$ acts trivially on $x$ and consider
$S=\ker(g-\id)$. Then $\im\fal{i}\subset S\otimes\hox$ and
$\fal{B}_k(S)\subset S\otimes\hox$ so $S=V$ and $g=\id$. If $x$ is
$\chi$-semistable let $y$ be a $\chi$-polystable ADHM datum such
that $(y,w)$ is in the closure of $G\cdot(x,z)$. Since
$\overline{G\cdot(x,z)}$ is disjoint from the the zero-section
(\cite[Lemma 2.2]{Ki}) we know that $w\neq 0$. Then $y$ is
$\chi$-stable and in particular it has a trivial stabilizer in
$G$. Therefore the orbit of $(y,w)$ has the maximal dimension. But
the set $\overline{G\cdot(x,z)}\backslash G\cdot(x,z)$ is composed
from the orbits of smaller dimension than the dimension of $G\cdot
(x,z)$. Therefore  $G\cdot(x,z)=G\cdot(y,w)$ and $x$ is also
$\chi$-stable.
\end{proof}

\medskip

\begin{Lemma} Let $x$ be an ADHM datum. Then $x$ is $1$-stable
(i.e., stable for the trivial character) if and only if it is
regular.
\end{Lemma}

\begin{proof}
Let us recall that $x$ is $1$-stable if and only if the stabilizer
of $x$ in $G$ is trivial and the orbit $G\cdot x$ is closed. Then
for any character and any $z\neq 0$ the orbit $G\cdot (x,z)$ is
closed as well. In particular, $x$ is both $\chi$-stable and
$\chi^{-1}$-stable, which by Proposition
\ref{rownowaznosc_stabilnosci} gives implication $\Rightarrow$.

Proof of the other implication is similar to the proof of Lemma
\ref{lemat_z_nakajimy}. Suppose that $x$ stable and costable and
$G\cdot x$ is not closed. There exists a one-parameter subgroup
$\lambda:\Gm\to G$ such that $\lim_{t\to 0}\lambda(t)\cdot x$
exists and belongs to $\overline{G\cdot x}\backslash G\cdot x$.
Let $V=\bigoplus_m V(m)$ be the weight decomposition with respect
to $\lambda$. As in proof of Lemma \ref{lemat_z_nakajimy}
existence of the limit $\lim_{t\to 0}\lambda(t)\cdot x$ implies
that
$$\fal{B_k}(V(m))\subset (\bigoplus_{l\geq m}V(l))\otimes\hox,$$
$$\im\fal{i}\subset (\bigoplus_{m\geq 0}V(m))\otimes\hox$$
and
$$ \bigoplus_{m> 0}V(m) \subset\ker\fal{j}.$$ The stability condition
implies that $V=\bigoplus_{m\geq 0}V(m)$ and costability gives
$\bigoplus_{m\geq 1}V(m)=\{0\}$. So $V=V(0)$ which contradicts our
assumption that the limit $\lim_{t\to 0}\lambda(t)\cdot x$ does
not  belong to $ G\cdot x$. The stabilizer of $x$ in $G$ is
trivial because $x$ is also $\chi$-stable.
\end{proof}

\medskip

\begin{Lemma}\label{polystable}
Let $x\in \ti\mu_{W,V}^{-1}(0)$. Then $x$ is $1$-polystable if and
only if there exist subspaces $V_1,V_2\subset V$ and quadruples
$x_1\in \ti\mu_{W,V_1}^{-1}(0)^{s,c}$ and $x_2\in \ti\mu_{\{ 0
\},V_2}^{-1}(0)$ such that $V=V_1\oplus V_2$, $x=x_1\oplus x_2$
and $\GL (V_2)\cdot x_2$ is closed. Moreover, such splitting is
unique.
\end{Lemma}

\begin{proof}
Let us remind that $x$ is $1$-polystable if and only if $\GL
(V)\cdot x$ is closed.

Assume first that $x=(\ti B_1,\ti B_2,\ti i, \ti j)$ has a closed
orbit and define $V_1$ as the intersection of all subspaces
$S\subset V$ such that $\ti B_k(S)\subset S\otimes H^0(\O_X(1))$
for $k=1,2$ and $\im \ti i \subset S\otimes H^0(\O_X(1))$. Choose
$V_2$ such that $V=V_1\oplus V_2$. Let $\{s_l\} $ be a basis of $
H^0(\O_X(1))$. Then $\ti B_k$, $\ti i$  and $\ti j$ can be written
as
$$\ti B_k =\sum B_{lk}\otimes s_l,\qquad \ti i=\sum i_{l}\otimes s_l, \qquad \ti j=\sum j_{l}\otimes s_l$$
where
$$B_{kl}=\left(\begin{array}{cc}
*&*\\
0&*
\end{array}\right),\qquad i_l=\left(\begin{array}{c}
*\\
0
\end{array}\right),\qquad j_l=\left(\begin{array}{cc}
*&*
\end{array}\right).$$
If $\lambda(t)=\left(\begin{array}{cc}
1&0\\
0&t^{-1}
\end{array}\right)$ then we have
$$\lambda(t)B_{kl}\lambda(t^{-1})=\left(\begin{array}{cc}
*&t*\\
0&*
\end{array}\right),\qquad \lambda(t)i_l=i_l,\qquad j_l\lambda(t^{-1})=\left(\begin{array}{cc}
*&t*
\end{array}\right).$$
Hence there exists $x'=(\ti B_1',\ti B_2',\ti i', \ti j')=
\lim_{t\rightarrow 0}\lambda(t)\cdot x$ which has the following
properties:
\begin{itemize}
\item $\ti B_k'(V_a)\subset V_a\otimes \hox$ for $k,a=1,2$,
\item $\ti B'_{k|V_1}=\ti B_{k|V_1}$ for $k=1,2$,
\item $\ti i'=\ti i,$
\item $\ti j'_{|V_2}=0,$
\item $\ti j'_{|V_1}=\ti j_{|V_1}.$
\end{itemize}

Since the orbit of $x$ is closed, we have $x'\in\GL (V)\cdot x$.
There exists $g\in\GL(V)$ such that $x'=g\cdot x$. So if we find
$V'_1, V'_2$ and $x'_1, x'_2$ satisfying conditions in the lemma
for $x'$, then $g\cdot V'_1,g\cdot V'_2$ and $g^{-1}\cdot
x'_1,g^{-1}\cdot x'_2$ satisfy it for $x$.

Let $V'_1$ be the intersection of all subspaces $S\subset V$ such
that $\ti B'_k(S)\subset S\otimes H^0(\O_X(1))$ for $k=1,2$ and
$\im \ti i' \subset S\otimes H^0(\O_X(1))$. Properties of $x'$
show that $V_1$ is one of such subspaces so $V'_1\subset V_1$. On
the other hand $g\cdot V'_1$ destabilizes $x$ so $V_1\subset
g^{-1}\cdot V'_1$ and by the dimension count we obtain $V'_1=V_1$.
Let us set $V'_2=V_2$ and $x'_1=(\ti B'_{1|V_1},\ti B'_{2|V_1},\ti
i', \ti j'_{|V_1})$, $x'_2=(\ti B'_{1|V_2},\ti B'_{2|V_2},0, \ti
j'_{|V_2})=(\ti B'_{1|V_2},\ti B'_{2|V_2},0, 0).$ It is clear that
$x'_1\in \ti\mu^{-1}_{W,V_1}(0)$ and $x'_2\in
\ti\mu^{-1}_{\{0\},V_2}(0)$ and $x'=x'_1\oplus x'_2$. Since $V'_1$
is minimal destabilizing space for $x'$, we also know that $x'_1$
is stable.

Now assume that $x''_1\in \ti\mu^{-1}_{W,V_1}(0)$ is in the
closure of the $\GL(V_1)$-orbit of $x'_1$ and $x''_2\in
\ti\mu^{-1}_{\{0\},V_2}(0)$ is in the closure of the
$\GL(V_2)$-orbit of $x'_2$. Then $x''_1\oplus x''_2$ is in the
closure of the $\GL(V)$-orbit of $x'=x'_1\oplus x'_2$. This orbit
is closed by the assumption so we can find $g\in\GL(V)$ such that
$g\cdot x'=x''_1\oplus x''_2$.

We can write
$$g=\left(\begin{array}{cc}
g_{11}&g_{12}\\
g_{21}&g_{22}
\end{array}
\right)$$
$$
x'_1=(\sum_l B'_{11l}\otimes s_l,\sum_l B'_{21l}\otimes s_l,\sum_l
i'_{1l}\otimes s_l,\sum_l j'_{1l}\otimes s_l)$$
$$
x'_2=(\sum_l B'_{12l}\otimes s_l,\sum_l B'_{22l}\otimes s_l,0,0)$$
$$x''_1=(\sum_l B''_{11l}\otimes s_l,\sum_l B''_{21l}\otimes s_l,\sum_l i''_{1l}\otimes s_l,\sum_l j''_{1l}\otimes s_l)$$
$$
x''_2=(\sum_l B''_{12l}\otimes s_l,\sum_l B''_{22l}\otimes
s_l,0,0)$$
$$x'=\left(\left(\begin{array}{cc}
B'_{11l}&0\\
0&B'_{12l}
\end{array}\right)\otimes s_l,\left(\begin{array}{cc}
B'_{21l}&0\\
0&B'_{22l}
\end{array}\right)\otimes s_l,\left(\begin{array}{c}
i'_{1l}\\
0
\end{array}\right)\otimes s_l,\left(\begin{array}{cc}
j'_{1l}&0
\end{array}\right)\otimes s_l\right)
$$
$$x''_1\oplus x''_2=\left( \left(\begin{array}{cc}
B''_{11l}&0\\
0&B''_{12l}
\end{array}\right)\otimes s_l,\left(\begin{array}{cc}
B''_{21l}&0\\
0&B''_{22l}
\end{array}\right)\otimes s_l,\left(\begin{array}{c}
i''_{1l}\\
0
\end{array}\right)\otimes s_l,\left(\begin{array}{cc}
j''_{1l}&0
\end{array}\right)\otimes s_l \right)
$$
The equality $g\cdot x'=x''_1\oplus x''_2$ gives us for each $l$
and $k=1,2$ the following equalities: {\setlength\arraycolsep{2pt}
\begin{eqnarray*}
\left(\begin{array}{cc}
g_{11}B'_{k1l}&g_{12}B'_{k2l}\\
g_{21}B'_{k1l}&g_{22}B'_{k2l}
\end{array}
\right)&=&\left(\begin{array}{cc}
g_{11}&g_{12}\\
g_{21}&g_{22}
\end{array}
\right)\cdot \left(\begin{array}{cc}
B'_{k1l}&0\\
0&B'_{k2l}
\end{array}\right)=\\
&=&\left(\begin{array}{cc}
B''_{k1l}&0\\
0&B''_{k2l}
\end{array}\right)\cdot \left(\begin{array}{cc}
g_{11}&g_{12}\\
g_{21}&g_{22}
\end{array}
\right)= \left(\begin{array}{cc}
B''_{k1l}g_{11}&B''_{k1l}g_{12}\\
B''_{k2l}g_{21}&B''_{k2l}g_{22}
\end{array}
\right)\\
\left(\begin{array}{c}
g_{11}i'_{1l}\\
g_{21}i'_{1l}
\end{array}
\right)&=& \left(\begin{array}{cc}
g_{11}&g_{12}\\
g_{21}&g_{22}
\end{array}
\right)\cdot \left(\begin{array}{c}
i'_{1l}\\
0
\end{array}\right)=
\left(\begin{array}{c}
i''_{1l}\\
0
\end{array}\right)\\
\left(\begin{array}{cc} j''_{1l}g_{11}&j''_{1l}g_{12}
\end{array}
\right) &=& \left(\begin{array}{cc} j''_{1l}&0
\end{array}\right)
\cdot \left(\begin{array}{cc} g_{11}&g_{12}\\ g_{21}&g_{22}
\end{array}
\right)= \left(\begin{array}{cc} j'_{1l}&0
\end{array}\right)
\end{eqnarray*}}

Let $S=\ker g_{21}\subset V_1$. The equalities above show that for
every $l$ and $k=1,2$ we have
$$B'_{k1l}(S)\subset S\qquad \im i'_{1l}\subset S,$$
which shows that $S$ is a destabilizing space for $x'_1$. Since
$x_1$ is stable $S=V_1$ and $g_{21}=0$. Therefore $g_{11}$ and
$g_{22}$ are isomorphisms.  Hence $$g_{11}\cdot x'_1=x''_2\quad
\textrm{and}\quad g_{22}\cdot x'_2=x''_2$$ which shows that the
orbits $\GL(V_1)\cdot x'_1$ and $\GL(V_2)\cdot x'_2$ are closed.
By the same argument as in Proposition
\ref{rownowaznosc_stabilnosci} we show that $x_1$ has a trivial
stabilizer in $GL(V_1)$. Hence $x_1$ is $1$-stable and by Lemma
\ref{polystable} it is costable.

\medskip
%%%%%%%%%%%%%%%%%%%%%%%%%%%%%%%%%%%
%%%   Implikacja w druga strone %%%
%%%%%%%%%%%%%%%%%%%%%%%%%%%%%%%%%%%

\begin{Remark} \label{CharacterisationOfDecomposition}
Note that the decomposition $V=V_1\oplus V_2$ is unique since
$V_1$ (respectively $V_2$) is the smallest (respectively the
biggest) subspace of $V$ such that
$$\ti B_k(V_1)\subset V_1\otimes\hox\textrm{ for }k=1,2\textrm{ and }\im\ti i\subset V_1\otimes\hox$$
$$(\textrm{respectively }\ti B_k(V_2)\subset V_2\otimes\hox\textrm{ for }k=1,2\textrm{ and }V_2\subset \ker\ti j).$$
Obviously, the splitting $x=x_1\oplus x_2$ is also unique.
\end{Remark}

\medskip

Now let us prove the opposite implication $\Leftarrow$. Fix
$x=(\ti B_1,\ti B_2,\ti i, \ti j)$ admitting a splitting
$x=x_1\oplus x_2$ as in the statement of the lemma. Let $Y$ be the
unique closed orbit contained in the closure of $\GL(V)$-orbit of
$x$. By \cite[Theorem 3.6]{Na1} there exists $x^0\in Y$ and
$\lambda:\Gm \rightarrow \GL(V)$ such that $\lim_{t\rightarrow
0}\lambda(t)\cdot x=x^0$. The implication proved above shows that
there exists a unique splitting $x^0=x_1^0\oplus x_2^0$ and
$V=V_1^0\oplus V_2^0$ as in the lemma. Put
$$\overline{V}^0_k=\lim_{t\rightarrow 0}\lambda(t)V_k\textrm{ and }\overline{x}^0_k=\lim_{t\rightarrow 0}\lambda(t)\cdot x_k\textrm{ for }k=1,2.$$
The first limit exists because subspaces in $V$ of fixed dimension
are parameterized by Grassmanians which are projective. The
remaining limits are restrictions of $x^0$ to $\overline{V}^0_1$
and $\overline{V}^0_2$, respectively.

\medskip

\begin{Remark}
\label{invariance_of_limit_space} Let us note that
$$\lambda(s)\lim_{t\rightarrow 0}\lambda(t)(V_k)=\lim_{t\rightarrow 0}\lambda(ts)(V_k)=\overline{V}^0_k.$$
Hence for any $s\in \Gm$ and $k=1,2$ we have
$\lambda(s)(\overline{V}^0_k)=\overline{V}^0_k$.
\end{Remark}

\medskip

First, let us suppose that $\lambda(t)(V_1)=V_1$ for all $t$. Then
$\overline{V}^0_1=V_1$ and for $k,l=1,2$ we have
\begin{eqnarray*}
(\lim_{t\rightarrow 0}\lambda(t)\cdot\ti B_l)(\overline{V}^0_k))&=&(\lim_{t\rightarrow 0}\lambda(t)\cdot\ti B_l\cdot\lambda(t^{-1}))(\lim_{s\rightarrow 0}\lambda(s)V_k)=\lim_{t\rightarrow 0}\lambda(t)\cdot\ti B_l(V_k)\subset\\
&\subset&\lim_{t\rightarrow
0}\lambda(t)(V_k\otimes\hox)=\overline{V}^0_k\otimes\hox
\end{eqnarray*}
$$\im(\lim_{t\rightarrow 0}\lambda(t)\ti i)= \lim_{t\rightarrow 0}\lambda(t)\im\ti i\subset \lim_{t\rightarrow 0}\lambda(t)V_1\otimes\hox=\overline{V}^0_1\otimes\hox$$
and
$$\ker\lim_{t\rightarrow 0}\ti j\lambda(t^{-1})=\lim_{t\rightarrow 0}\ker\ti j\lambda(t^{-1})\supset \lim_{t\rightarrow 0}\lambda(t)V_2=\overline{V}^0_2.$$
Therefore by the characterization of $V^0_1$ and $V^0_2$ given in
Remark \ref{CharacterisationOfDecomposition} we have
$V^0_1\subset\overline{V}^0_1$ and $\overline{V}^0_2\subset
V^0_2$. Since $\lambda(t)$ preserves $V_1$ for all $t$ we see that
$\overline{x}^0_1\in\overline{\GL(V_1)\cdot x_1}=\GL(V_1)\cdot
x_1$. This shows that $\overline{x}^0_1$ is stable and again using
Remark \ref{CharacterisationOfDecomposition} we obtain equality
$V_1^0=\overline{V}^0_1$. Then the dimension count shows that
$V_2^0=\overline{V}^0_2$ and in particular
$V=\overline{V}^0_1\oplus\overline{V}^0_2$.

Consider the unipotent group
$$U_{\lambda}=\{u\in\GL(V)|\lim_{t\rightarrow 0}(\lambda(t)u\lambda(t)^{-1})=1\}.$$
Observe that we can replace $\lambda$ by $\lambda'=u\lambda
u^{-1}$ for $u\in U_\lambda$. Indeed, one can easily prove that
$\lim_{t\rightarrow 0}\lambda'(t)x=ux^0$ and $ux^0$ represents the
same orbit as $x^0$.

We will show that there exists $u\in U_{\lambda}$ such that
\begin{equation}
u(\overline{V}^0_k)=V_k\textrm{ for }k=1,2\label{condition_u} .
\end{equation}
Set $m=\dim\overline{V}^0_1$ and $n=\dim V$. By Remark
\ref{invariance_of_limit_space} one can choose a basis
$\A=(\alpha_i)_{i=1}^n$ of $V$ such that $(\alpha_i)_{i=1}^m$ is a
basis of $\overline{V}^0_1$, $(\alpha_i)_{i=m+1}^n$ is a basis of
$\overline{V}^0_2$ and $\lambda(t)\alpha_i=t^{a_i}\alpha_i$ for
some $a_i\in \ZZ$ and all $t\in \Gm$. Let $(u_{ij})$ be the matrix
of $u\in \GL(V)$. Then
\begin{equation}
(u_{ij})\in U_{\lambda}\iff u_{ii}=1\textrm{ and }u_{ij}=0\textrm{
for }i\neq j\textrm{ such that }a_i\le a_j.\label{condition_uij}
\end{equation}
Let $\B=(\beta_j)_{j=m+1}^n$ be a basis of $V_2$ such that
$$\beta_j=\alpha_j+\sum_{i=1}^m c_{ij}\alpha_i\textrm{ for }j=(m+1), \ldots,n.$$ Such a basis
exists because $\dim V_2=\dim \overline{V}^0_2$ and $V_2\cap
\overline{V}^0_2=\{ 0\}$. Let $b_j=\min\{a_i:i=j\textrm{ or
}c_{ij}\neq 0\}$. Then
$$\lim_{t\rightarrow 0}\lambda(t)\lin(\beta_j)=\lim_{t\rightarrow
0}\lin(t^{a_j}\alpha_j+\sum_{i=1}^m c_{ij}t^{a_i}\alpha_i)
=\lin(\lim_{t\rightarrow 0}(t^{a_j-b_j}\alpha_j+\sum_{i=1}^m
c_{ij}t^{a_i-b_j}\alpha_i)).$$

The vector $\lim_{t\rightarrow 0}(t^{a_j-b_j}\alpha_j+\sum_{i=1}^m
c_{ij}t^{a_i-b_j}\alpha_i)$ exists by the  definition of $b_i$.
Moreover, since $\lim_{t\rightarrow
0}\lambda(t)V_2=\overline{V}^0_2$, it must be contained in
$\overline{V}^0_2$. Therefore $b_j=a_j$ and $c_{ij}=0$ when
$a_i\le a_j$. Now, with respect to the basis $\A$, we define
$u\in\GL(V)$ by matrix $(u_{ij})$ with the following coefficients:
$$u_{ij}=\left\{
\begin{array}{ll} 1 &\textrm{ for } i=j,\\
c_{ij}& \textrm{ for }i\le m\textrm{ and }j\ge m+1,\\
0 &\textrm{ in the remaining cases.}
\end{array}
\right.
$$

By (\ref{condition_uij}) such $u$ belongs to $U_{\lambda}$ and
satisfies (\ref{condition_u}).

Using Remark \ref{invariance_of_limit_space} we get for $k=1,2$
$$\lambda'(t)(V_k)=u\lambda(t)                                                                                                                                                                                                                                                                                                                                                                                                                                                                                                                                                                                                                                                                                                                                                                                                                                                                                                                                                                                                                                                                                                                                                                                                                                                                                                                                                                                                                                                                                                                                                                                                                                                                                                                                                                                                                                                                                                                                                                                                                                                                                                                                                                                                                                                                                                                                                                                                                                                                                                                                                                                                                                                                                                                                                                                                                                                                                                                                                                                                                                                                                                                                                                                                                                                                                                                                                                                                                                                                                                                                                                                                                                                                                                                                                                                                                                                                                                                                                                                                                                                                                                                                                                                                                                                                                                                                                                                                                                                                                                                                                                                                                                                                                                                                                                                                                                                                                                                                                                                                                                                                                                                                                                                                                                                             u^{-1}(V_k)=u\lambda(t)(\overline{V}^0_k)=u(\overline{V}^0_k)=V_k.$$
We can therefore assume that $\lambda=\lambda_1\times \lambda_2$,
where $\lambda_a$, $a=1,2,$ is a one-parameter subgroup of
$\GL(V_a)$. Then, since $x_1$, $x_2$ have closed orbits, we get
$$x^0=\lim_{t\rightarrow 0} (\lambda_1(t)\cdot x_1\oplus
\lambda_2(t)\cdot x_2)\in \GL(V_1)\cdot x_1\oplus\GL(V_2)\cdot
x_2\subseteq \GL(V)\cdot x, $$ which proves that the orbit of $x$
is also closed.

\medskip

In general, by similar arguments as above there exists an element
$u\in U_{\lambda}$ such that $u(\overline{V}^0_1)=V_1$. Consider
the one-parameter subgroup $\lambda'=u\lambda u^{-1}$. Then we
have $\lambda'(t)(V_1)=V_1$ for all $t$ and $\lim_{t\rightarrow
0}(\lambda'(t)\cdot x)=ux^0.$ Thus, the previous part of the proof
implies that $ux^0\in\GL(V)\cdot x$ and it proves that the orbit
of $x$ is closed.
\end{proof}

\medskip

\begin{Remark}
Note that up to now we have never assumed that either $r$ or $c$
is positive. In fact, $r=0$ is very interesting due to Lemma
\ref{polystable}. This lemma shows that at least
set--theoretically one can reduce the study of $1$-polystable ADHM
data to regular ADHM data and $1$-polystable ADHM data in the rank
$0$ case. We explain the geometric meaning of this fact in the next section.

Note that in case of rank $0$ there are no stable ADHM
data so $\ti\mu ^{-1}(0)/\!\!/_{\chi}G=\emptyset$. But the
quotient $\ti\mu ^{-1}(0)/G$ is still a highly non-trivial scheme.
\end{Remark}

\section{Gieseker and Donaldson--Uhlenbeck partial
compactifications of instantons}

In this section we consider ADHM data for $X=\PP ^1$, which by
Theorem \ref{bijection-ADHM-perverse} correspond to perverse
instantons on $\PP ^3$. In this case we obtain a similar picture
as that known from framed torsion free sheaves on $\PP ^2$ (see
\ref{section:P^2}).

\begin{Definition}
A perverse instanton $\C $ is called \emph{stable} (\emph{costable, regular})
if it comes from some stable (respectively: costable, regular) ADHM datum.
\end{Definition}

\medskip
Let us recall that we have a natural action of $G=\GL (V)$ on the set $\ti \mu^{-1}(0)$
of ADHM data. This action induces an action on the open subset
$\ti\mu^{-1}(0)^{s}$ of stable ADHM data, which by Lemma \ref{rownowaznosc_stabilnosci}
corresponds to $\chi$-stable points for the character $\chi: G\to
\Gm$ given by the determinant. The proof of Lemma \ref{rownowaznosc_stabilnosci}
shows that $\mu^{-1}(0)^s\to \mu^{-1}(0)^s/G$ is a principal $G$-bundle in the
\'etale topology (In fact, in positive characteristic we
also need to check scheme-theoretical stabilizers. Then the assertion
follows from a version of Luna's slice theorem. We leave the details to the reader.)

\medskip

\medskip
Let $\overline{\M}(\PP ^3; r,c): \Sch/k\to \Sets$ be the functor
which to a scheme $S$ assigns the set of isomorphism classes of
$S$-families of stable framed perverse $(r,c)$-instantons.
Theorem \ref{perv-theorem} and the above remarks imply that this
functor is representable:

\begin{Theorem}\label{perv-moduli}
The quotient $\M(\PP ^3; r,c):=\ti\mu^{-1}(0)/\!\!/_{\chi} G$ is a
fine moduli scheme for the functor $\overline{\M}(\PP ^3;
r,c)$. In particular, there is a bijection between $G$-orbits of
stable ADHM data and isomorphism classes of stable framed
perverse instantons.
\end{Theorem}

\medskip

Since every FJ-stable ADHM datum is stable we get as a corollary
the following theorem generalizing the main theorem of \cite{FJ}:

\begin{Theorem}
Let $\ti\mu^{-1}(0)^{FJ}$ be the set of FJ-stable ADHM data. Then
the  GIT quotient $\M ^{f}(\PP ^3;
r,c):=\ti\mu^{-1}(0)^{FJ}/\!\!/_{\chi} G$ represents  the moduli
functor of rank $r$ instantons  on $\PP ^3$ with $c_2=c$, framed
along a line $l_{\infty}$.  In particular, there is a bijection
between $G$-orbits of FJ-stable ADHM data and isomorphism classes
of framed $(r,c)$-instantons. Moreover, orbits of FJ-regular ADHM
data are in bijection with isomorphism classes of locally free
instantons.
\end{Theorem}

Let $\M _0^{\reg} (\PP ^3; r,c)$ be the moduli space of regular
framed perverse $(r,c)$-instantons. By Theorem \ref{perv-moduli}
$\M _0^{\reg} (\PP ^3;r,c)$ is isomorphic to the quotient of
regular ADHM data by the group $G$. The space $\M(\PP ^3; r,c)$
contains the moduli space $\M _0^{\reg} (\PP ^3;r,c)$ as an open
subset and it can be considered as its partial Gieseker
compactification. Note also that FJ-semiregular ADHM data are
regular, so $\M _0^{\reg} (\PP ^3; r,c)$ contains the moduli space
of framed reflexive $(r,c)$-instantons as on open subset.

Let $\M _0(\PP ^3;r,c)$ denote the quotient $\ti\mu ^{-1} (0)/G$.
This is an affine scheme and it contains the moduli space $\M
_0^{\reg} (\PP ^3;r,c)$ as an open subset. It can be considered as
its partial Donaldson--Uhlenbeck compactification.

\begin{Proposition} \label{interpretation}
For every stable rank $r>0$ perverse instanton $\C$  on $\PP ^3$
there exists a regular rank $r$ perverse instanton $\C'$ and a
rank $0$ perverse instanton $\C ''$ such that we have a
distinguished triangle
$$\C''\to \C\to \C'\to \C''[1].$$
\end{Proposition}

\begin{proof}
Fix a stable ADHM datum $x=(\fal{B}_1,\fal{B}_2,\fal{i},\fal{j} )
\in \ti\mu^{-1}(0)^s$ corresponding to a perverse instanton $\C $
(see Theorem \ref{bijection-ADHM-perverse}). Then by \cite[Theorem
3.6]{Na1} there exists a one-parameter subgroup $\lambda:\Gm
\rightarrow\GL(V)$ such that
$x^0=\lim_{t\rightarrow0}(\lambda(t)\cdot x)$ exists and it is
contained in the unique closed orbit in $\overline{\GL(V)\cdot
x}$. Let us set
$x^0=(\fal{B}^0_1,\fal{B}^0_2,\fal{i}^0,\fal{j}^0)$ and  fix a
splitting $x^0=x^0_1\oplus x^0_2$, $V=V_1\oplus V_2$ as in Lemma
\ref{polystable}.

As before we can consider the weight decomposition
$$V=\bigoplus_{m\in \ZZ} V (m) \textrm{, where }V (m)=\{v\in
V|\lambda(t)\cdot v=t^m v\}.$$ Since $x$ is stable we have
$V=\bigoplus_{m\ge 0}V (m)$. Let $i^0$ be the composition of $i$
and the natural projection $p_1:V_1\oplus V_2\rightarrow V_1$. We
claim that
$$V_1=V (0),\quad x^0_1=(\fal{B}^0_{1|V_1},\fal{B}^0_{2|V_1},\fal{i}^0,\fal{j}),$$
$$V_2=\bigoplus_{m\ge 1}V (m),\quad x^0_2=(\fal{B}^0_{1|V_2},\fal{B}^0_{2|V_2},0,0).$$

By Remark \ref{CharacterisationOfDecomposition} it is enough to
show that $V (0)$ is the smallest destabilizing subspace for $x^0$
and $\bigoplus _{m\ge 1}V (m)$ is the biggest subspace
"decostabilizing" $x^0$. It is easy to see that $V (0)$ indeed
destabilizes $x^0$. If there was a proper subspace $S\subset V
(0)$ with the same property then $S\oplus \bigoplus _{m\ge 1}V
(m)$ would destabilize $x$. A similar argument applies to
$\bigoplus _{m\ge 1}V (m)$.

Varagnolo and Vasserot in \cite[proof of Theorem 1]{VV} claimed
that $x^0_2=(\fal{B}_{1|V_2},\fal{B}_{2|V_2},0,0)$. In our case
this equality does not hold.  Let us set
$x_2=(\fal{B}_{1|V_2},\fal{B}_{2|V_2},0,0)$. Since
$\fal{j}_{|V_2}=0$ one can easily see that $x_2$ satisfies the
ADHM equation and $x_2 \in \ti\mu^{-1}_{0,V_2}(0)$.

Although $x$ and $x^0_1\oplus x_2$ in general are not equal, we
still have the following  exact triple of complexes:
\begin{equation}
0\rightarrow \C ^{\bullet}_{x_2}\rightarrow \C ^{\bullet}_x
\rightarrow \C ^{\bullet}_{x^0_1}\rightarrow 0.
\end{equation}
This triple gives rise to the required distinguished triangle.
\end{proof}

\medskip
As a corollary to the above proposition we can describe the
morphism from Gieseker to Donaldson--Uhlenbeck partial
compactifications of $\M _0^{\reg} (\PP ^3;r,c)$. Namely, we have
a natural set-theoretical decomposition
$$\M _0(\PP ^3;r,c)=\bigsqcup _{0\le d \le c} \M _0^{\reg} (\PP ^3;r,c-d)\times \M _0 (\PP ^3;0,d).$$
Then the natural morphism
$$\M (\PP ^3;r,c)\simeq \ti \mu ^{-1} (0)/\!\!/_{\chi} G\to \ti \mu^{-1} (0)/ G \simeq \M _0(\PP ^3;r,c)$$
coming from the GIT (see Subsection \ref{GIT-section}) can be
identified with the map
$$(\C, \Phi)\to ((\C', \Phi '), \C ''),$$
where $\C'$ and $\C''$ are as in Proposition \ref{interpretation}
and $\Phi'$ is induced on $\C'$ via $\Phi$. This morphism is
analogous to the one described in Subsection \ref{section:P^2}.

\medskip

\begin{Proposition}
For every framed rank $r>0$ instanton $E$ on $\PP ^3$ there exists
a unique regular rank $r$ instanton $E'$ containing $E$. Moreover,
the inclusion map $E\to E'$ is uniquely determined and we have a
short exact sequence
$$ 0\to E\to E'\to E''\to 0,$$
where $E''$ is a rank $0$ instanton (see Definition
\ref{rank-zero}).
\end{Proposition}

\begin{proof}
Let us consider the short exact sequence from the proof of previous proposition.
Since $\C^{\bullet}_{x_2}$ is a rank $0$ perverse instanton, we have $\H^0(\C ^{\bullet}_{x_2})= 0$.
Thus we obtain the following long exact sequence of cohomology groups
$$
0\rightarrow\H^0(\C ^{\bullet}_x) \rightarrow
\H^0(\C^{\bullet}_{x^0_1})\rightarrow  \H^1(\C
^{\bullet}_{x_2})\rightarrow\H^1(\C ^{\bullet}_x) \rightarrow
\H^1(\C^{\bullet}_{x^0_1})\rightarrow 0.
$$
By Lemma \ref{FJmonada} $x$ is FJ-stable if and only if $\H^1(\C
^{\bullet}_x)=0$. In particular, if $x$ is FJ-stable then $\H^1(\C
^{\bullet}_x)=0$. This implies that $\H^1(\C
^{\bullet}_{x^0_1})=0$ and hence $x^0_1$ is also FJ-stable.
Therefore we can set $E'= \H^0(\C^{\bullet}_{x^0_1})$ and $E''=
\H^1(\C ^{\bullet}_{x_2})$. Let us set $c'=\dim V_1$. Our choice
of $x^0_1$ and $x_2$ shows that $E'$ is a torsion free
$(r,c')$-instanton corresponding to a costable ADHM datum, and
$E''$ is the first cohomology of a perverse $(0,c-c')$-instanton
(let us recall that such instantons have no other non-trivial
cohomology).
\end{proof}

\medskip

The above proposition allows us to describe the morphism from the
moduli space  $\M ^{f}(\PP ^3; r,c)$ of framed instantons to the
Donaldson--Uhlenbeck partial compactification of $\M _0^{\reg}
(\PP ^3;r,c)$.

\section{Perverse instantons of rank $0$}

In this section we describe the moduli space $\M _0 (\PP ^3;0,c)$
of perverse instantons of rank $0$. Let us recall that this
``moduli space'' does not corepresent any functor and in
particular, as for Chow varieties, we do not have any deformation
theory. But we can still show that closed points of this moduli
space can be interpreted as certain $1$-dimensional sheaves on
$\PP ^3$. Then we relate the moduli space to modules over a
certain non-commutative algebra and we show that already $\M _0
(\PP ^3;0,2)$ is reducible.

\medskip

\subsection{Rank $0$ instantons}

\begin{Definition} \label{rank-zero}
A  \emph{rank $0$ instanton} $E$ on $\PP ^3$ is a pure sheaf   of
dimension $1$ such that $H^0(\PP ^3 , E(-2))=0$ and $H^{1}(\PP^3,
E(-2))=0$.
\end{Definition}

The above definition is motivated by the following lemma:

\begin{Lemma}
If $\C$ is a rank $0$ perverse instanton then $\C [1] $ is a sheaf
object whose underlying sheaf is a rank $0$ instanton. On the
other hand, if $E$ is a rank $0$ instanton then the object $E[-1]$
in $D^b(\PP ^3)$ is a rank $0$ perverse instanton.
\end{Lemma}

\begin{proof}
If $\C$ is a rank $0$ perverse instanton then only $E=\H^1(\C)$ is
non-zero and hence $\C [1]$ is a sheaf object. Clearly, it has
dimension $\le 1$, since there exists a line $l$ such that the
support of $E$ does not intersect $l$. Since
$$H^p(\PP^3, \C \otimes \O_{\PP^3}(q))= H^{p-1}(\PP^3, E(q))$$
we see the required vanishing of cohomology. To prove that $E$ is
pure of dimension $1$ note that the torsion in $E$ would give a
section of $H^0(\PP ^3 , E(-2))$. This proves the first part of
the lemma.

Now assume that $E$ is a rank $0$ instanton and set $\C=E[-1]$.
Conditions 2 and 3 from Definition \ref{perv-definition} are
trivially satisfied for $\C$. To check the condition $1$ it is
sufficient to prove that $H^0(\PP ^3 , E(q))=0$ for $q\le -2$ and
$H^{1}(\PP^3, E(-2))=0$ for $q\ge -2$. By \cite[Lemma 1.1.12]{HL2}
there exists an $E(m)$-regular section of $\O_{\PP ^3}(1)$ and it
gives rise to the sequence
$$0\to E(m-1)\to E(m) \to E'\to 0$$
in which  $E'$ is some sheaf of dimension $0$. Using such
sequences and the definition of rank $0$ instanton it is easy to
check the required vanishing of cohomology groups.
\end{proof}

\medskip
By definition closed points of $\M _0 (\PP ^3;0,d)$ correspond to
closed $\GL (c)$-orbits of ADHM $(0,c)$-data for $\PP ^1$. By
Theorem \ref{bijection-ADHM-perverse} and the above lemma there
exists a bijection between isomorphism classes of rank $0$
instantons $E$ whose scheme-theoretical support is a curve od
degree $c$ not intersecting $l_{\infty}$ and $\GL (c)$-orbits of
ADHM $(0,c)$-data for $\PP ^1$.  So $\M _0 (\PP ^3;0,d)$ can be
thought of as the moduli space of some pure sheaves of dimension
$1$. Note however that this moduli space is only set-theoretical
and it is not a coarse moduli space.

\medskip

In the characteristic zero case $\ti\mu ^{-1}(0)/G $ is a
subscheme of the quotient $\ti B/G$, which is a normal variety.
Moreover, the coordinate ring for the variety $\ti B/G$ can be
described using the First Fundamental Theorem for Matrices (see
\cite[2.5, Theorem]{KP}). More precisely, if $\char k=0$ then
$$k[\ti B/G]=k[\ti B]^{G}=k[{\Tr}_{i_1\dots i_m} : 1\le i_1,\dots, i_m\le 4 , m\le c^2],$$
where $\Tr _{i_1\dots i_m}: \ti B\simeq\enD (V) ^4\to k$ is the
\emph{generalized trace} defined by
$$(A_1,A_2,A_3,A_4)\to \Tr (A_{i_1}A_{i_2}\dots A_{i_m}).$$
This in principle allows us to find $\ti\mu ^{-1}(0)/G$ as the
image of $\ti\mu ^{-1}(0)$ in  $\ti B/G$. In practice, computer
assisted computations using this interpretation almost never work
due to complexity of the problem.

\medskip

\medskip

\subsection{Schemes of modules over an associative ring}

In this subsection we recall a construction of the moduli space of
$d$-dimensional modules over an associative ring. It is mostly a
folklore, but note that our moduli space is not the same as the
one constructed by King in \cite{Ki}. We are interested in the
moduli space that was introduced by Procesi in \cite{Pr} (see also
\cite{Mo} for a more functorial approach) but it is
non-interesting from the point of view of finite dimensional (as
$k$-vector spaces) algebras. In our treatment we restrict to the
simplest case although the constructions act in much more general
set-up.

\medskip

Let $k$ be an algebraically closed field and let $R$ be a finitely
generated associative $k$-algebra with unit. Let us fix a positive
integer $d$. Let $\Mod _R^d$ denote the \emph{scheme of
$d$-dimensional $R$-module structures}. By definition it is the
affine (algebraic) $k$-scheme representing the functor from
commutative $k$-algebras (with unit) to the category of sets
sending a $k$-algebra $A$ to
$$ {\Mod} _R^d(A)=\{\mbox{left $R\otimes _k A$-module structures on }A^d\}
=\{\mbox{$A$-algebra maps }R\otimes _k A\to {\Mat} _{d\times
d}(A)\},$$ where $ \Mat _{d\times d}(A)$ denotes the set of
$d\times d$-matrices with values in $A$.

Let us choose a surjective homomorphism $\pi: k \langle x_1,\dots
,x_n\rangle \to R$ from the free associative algebra with unit.
Then the above functor is naturally equivalent to the functor
sending $A$ to the set of $n$-tuples $(M_1,\dots ,M_n)$ of
$d\times d$-matrices with coefficients in $A$ such that
$f(M_1,\dots, M_n)=0$ for all $f\in \ker \pi$. In particular, the
$k$-points of $\Mod _R^d$ correspond to $R$-module structures on
$k^d$ (i.e., to $d$-dimensional $R$-modules with a choice of a
$k$-basis).

We have a natural $\GL (d)$-action on $\Mod _R^d$ which
corresponds to a change of bases (it gives the conjugation action
on the set of matrices). By the GIT, there exists a uniform good
quotient $Q _R^d =\Mod _R^d/\GL _d$.

Let us recall that if $S$ is a $k$-scheme then a \emph{family of
$d$-dimensional $R$-modules parameterized by } $S$ (or simply an
\emph{$S$-family of $R$-modules}) is a locally free coherent
$\O_S$-module $\F$ together with a $k$-algebra homomorphism $R\to
\enD \F$.

\begin{Proposition}
The quotient  $Q _R^d$ corepresents the moduli functor
 $\Q _R^d: \Sch/k\to \Sets$ given by
$$S\to\left\{
\hbox{Isomorphism classes of $S$-families of $R$-modules} \\
\right\} .
$$
We call it the \emph{moduli space of $d$-dimensional $R$-modules}.
\end{Proposition}

Proof of this proposition is completely standard and we leave it
to the reader (cf. \cite[Lemma 4.1.2]{HL2} and \cite[Proposition
5.2]{Ki})

The quotient $Q_R^d$ parameterizes closed $\GL_d$-orbits in $\Mod
_R^d$. An orbit of a $k$-point is closed if and only if it
corresponds to a semisimple representation of $R$. Therefore the
$k$-points of  $Q _R^d$ correspond to isomorphism classes of
$d$-dimensional semisimple $R$-modules. Equivalently, $Q _R^d$
parameterizes S-equivalence classes of $d$-dimensional
$R$-modules, where two modules are S-equivalent if the graded
objects associated to their Jordan--H\"older filtrations are
isomorphic.

Note that if $R$ is commutative then $Q_R^d$ is the moduli space
of zero-dimensional coherent sheaves of length $d$ on $X=\Spec R$.
Usually, the moduli spaces on non-projective varieties do not make
sense but in case of zero-dimensional sheaves we can take any
completion of $X$  to a projective scheme $\overline{X}$ and
consider the open subscheme of the moduli space of
zero-dimensional coherent sheaves of length $d$ on $\overline{X}$,
which parameterizes sheaves with support contained in $X$.

We will need the following proposition:

\begin{Proposition} \label{symmetric-power}
Let $R$ be a commutative $k$-algebra and let $X=\Spec R$. Then we
have a canonical morphism $f: S^dX\to Q _R^d$ from the $d$-th
symmetric power of $X$, which is a  bijection on the sets of
closed points. If $k$ is a field of characteristic zero then $f$
is an isomorphism.
\end{Proposition}

\begin{proof}
Let us consider the morphism $Q _R^1\times \dots \times Q _R^1\to
Q _R^d$ from the $d$ copies of $Q _R^1$, given by taking a direct
sum. Clearly, $Q _R^1=\Spec R$ and the morphism factors through
$S^d\AA^n$ as it is invariant with respect to the natural action
of symmetric group exchanging components of the product. The
induced morphism $S^d\AA^n\to Q _R^d$ is an isomorphism on the
level of closed $k$-points since a simple module over a
commutative algebra is $1$-dimensional (e.g., by Schur's lemma).

The second part follows from \cite[Example 4.3.6]{HL2} (note that
the proof works also if the  characteristic is sufficiently high)
and the interpretation of $Q^d_R$ that we gave above.
\end{proof}

\medskip

\begin{Remark}\label{strange}
We note in the next subsection that the scheme of pairs of
commuting $d\times d$-matrices is irreducible. But already the
scheme of triples of commuting $d\times d$-matrices (i.e., $\Mod
_R^d$ for $R=k[x_1,x_2,x_3]$) is reducible for $d\ge 30$ (see
\cite[Proposition 3.1]{HO}). Still the above proposition says that
its quotient $Q_R^d$ is irreducible if $R$ is commutative and
$\Spec R$ is irreducible.
\end{Remark}

\begin{Example} \label{pairs-of-matrices}
Let us consider $M_{k[x_1,x_2]}^2$.  Let us set
$$B_{1}=\left(\begin{array}{cc}
1&0\\
y_{1}&1
\end{array}\right)
\left(\begin{array}{cc}
y_{3}&y_{2}(y_{3}-y_{4})\\
0&y_{4}
\end{array}\right)
\left(\begin{array}{cc}
1&0\\
-y_{1}&1
\end{array}\right)
$$
and
$$B_{2}=\left(\begin{array}{cc}
1&0\\
y_{1}&1
\end{array}\right)
\left(\begin{array}{cc}
y_{5}&y_{2}(y_{5}-y_{6})\\
0&y_{6}
\end{array}\right)
\left(\begin{array}{cc}
1&0\\
-y_{1}&1
\end{array}\right).
$$
One can easily check that the condition $[B_{1},B_{2}]=0$ is
satisfied. Therefore we can define the map $\psi: \AA ^6\to
M_{k[x_1,x_2]}^2$ by sending $(y_1,\dots , y_6)$ to $(B_1,B_2)$.
By the previous remark $M_{k[x_1,x_2]}^2$ is irreducible  and one
can check that the above defined map is dominant and  generically
finite.

Let us recall that we have the map $\eta:
\AA^4=\AA^2\times\AA^2\to S^2\AA^2\to Q_{k[x_1,x_2]}^2$. One can
easily see that the image of a point $(y_1,\dots , y_6)\in \AA ^6$
in $Q_{k[x_1,x_2]}^2$ coincides with the image under $\eta$ of the
quadruple $(y_3,y_4,y_5,y_6)\in \AA^4$ consisting of pairs of
eigenvalues of matrices $(B_1,B_2)=\psi (y_1,\dots , y_6)$.

\end{Example}

\subsection{Moduli interpretation for instantons of rank $0$.}

Let us first consider the ADHM data for a point for $r=0$ and some
positive $c>0$. The moment map
$$\mu :{\bf B}= \enD (V)\oplus \enD (V)\to \enD (V)$$ is in this case
given by $(B_1,B_2)\to [B_1,B_2]$, where as usual $V$ is a
$k$-vector space of dimension $c$. In this case $\mu ^{-1}(0)$ is
known as the \emph{variety of commuting matrices}.  It is known to
be irreducible by classical results of Gerstenhaber \cite{Ge} and
Motzkin and Taussky \cite {MT}. This implies that the quotient
$\mu ^{-1}(0)/\GL (V)$ is also irreducible. In fact, one can see
from the definition that $\mu ^{-1}(0)/\GL (V)$ is isomorphic to
the scheme $Q _{k[x_1,x_2]}^c$ of equivalence classes of
$c$-dimensional $k[x_1,x_2]$-modules. Therefore by Proposition
\ref{symmetric-power} the points of $\mu ^{-1}(0)/\GL (V)$ are in
bijection with the points of $c$-th symmetric power $ S^c(\AA ^2)$
of $\AA ^2$ (this should be compared with \cite[Proposition
2.10]{Na1} which gives a different bijection). In characteristic
zero we get that $\mu ^{-1}(0)/\GL (V)$ is isomorphic to $S^c \AA
^2$.

\medskip

Now let us consider ADHM data for $\PP ^1$. Again it follows from
the definitions that the quotient $\M _0 (\PP ^3;0,c)=\ti \mu
^{-1}(0)/\GL (V)$ is isomorphic to the scheme $Q _R^c$ of
equivalence classes of $c$-dimensional $R$-modules for a
non-commutative $k$-algebra
$$R=k\langle y_1,y_2,z_1,z_2\rangle /(y_1y_2-y_2y_1, z_1z_2-z_2z_1,
y_1z_2-z_2y_1+y_2z_1-z_1y_2).$$ Let us define a two--sided ideal
in $R$ by
$$I=(y_1z_2-z_2y_1, y_1z_1-z_1y_1, y_2z_2-z_2y_2).$$
It is easy to see that $R/I\simeq k[y_1,y_2,z_1,z_2]$ so we have a
surjection $R\to R'= k[y_1,y_2,z_1,z_2]$. This induces a closed
embedding of affine schemes
$${\Mod} _{R'}^c\subset {\Mod} _R^c$$
(see \cite[Proposition 1.2]{Mo}). Therefore we get a morphism
$$ Q_{R'}^c\to Q _R^c,$$
which is a set--theoretical injection of quotients. If $k$ has
characteristic zero then this morphism is a closed embedding.

By Proposition \ref{symmetric-power} we get the following induced
affine map
$$\varphi: S^c\AA ^4\to Q_R^c\simeq \M _0 (\PP ^3;0,c),$$
which is a set-theoretical injection.

Geometric interpretation of the map $\varphi$ is the following.
Note that $\AA ^4$ parameterizes the lines in $\PP ^3$ that do not
intersect $l_{\infty}$. Then for a point in $S^c\AA ^4$, the image
corresponds to the rank $0$ instanton $E=\O _{l_1}(1)\oplus \dots
\oplus \O_{l_c}(1)$, where $l_1, \dots ,l_c$ are the lines not
intersecting $l_{\infty}$.  If all these lines are disjoint then
the corresponding rank $0$ instanton $E$ gives a point in the
Hilbert scheme of curves of degree $c$ and one can check that the
corresponding component has dimension $4c$. This suggest the
following proposition:

\begin{Proposition}
The image of $\varphi$ is an irreducible component of $ \M _0 (\PP
^3;0,c)$ of dimension $4c$.
\end{Proposition}

\begin{proof}
Let $l_1, \dots ,l_c$ be disjoint lines not intersecting
$l_{\infty}$. Let $E=\O _{l_1}(1)\oplus \dots \oplus \O_{l_c}(1)$
be the corresponding rank $0$ instanton $E$ and let $x\in {\ti
\mu}^{-1} (0)={\Mod} _R^c$ be an ADHM datum corresponding to $E$.
Let $X$ denote the $R$-module corresponding to $x$.

By Theorem \ref{deformation}  there exists a surjective map
$T_x{\Mod} _R\to \ext ^1_{\PP^3}(E,E)$ whose kernel is the tangent
space of the orbit $\O (x)$ of $x$ at $x$. The support of $E$ does
not intersect $l_{\infty}$ so we do not need to tensor by
$J_{l_{\infty}}$. Note that the theorem (and its proof) still
works in our case but in the formulation given above: $d \varphi
_e$ is not injective. Let us also note that this fact, together
with Voigt's theorem, shows that $\ext ^1_R(X,X)\simeq \ext
^1_{\PP^3}(E,E)$.

\begin{Lemma}
Let $l$ be any line in $\PP ^3$. Then $\dim \ext ^1_{\PP^3}(\O_l,
\O_l)=4$.
\end{Lemma}

\begin{proof}
Using the short exact sequence
$$0\to J_{l}\to \O_{\PP ^3}\to \O _l\to 0$$
we see that $\ext ^1_{\PP^3}(\O_l, \O_l)\simeq \hom (J_{l},
\O_l)$. To compute this last group we can assume that $l$ is given
by equations $x_2=x_3=0$. Then we have a short exact sequence
$$0\to \O_{\PP ^3} (-2)\mathop{\to}^{(x_2,x_3)} \O_{\PP ^3}(-1)^2\to J_l\to 0$$
which gives an exact sequence
$$0\to \hom (J_{l}, \O_l) \to \hom ( \O_{\PP ^3}(-1)^2, \O _l)\mathop{\to}^{f}  \hom ( \O_{\PP ^3}(-2), \O _l).$$
Since $f$ is the zero map, we see that $ \hom (J_{l}, \O_l) \simeq
\hom ( \O_{\PP ^3}(-1)^2, \O _l)\simeq H^0(\O_l(1))^{\oplus 2}$ is
$4$-dimensional.
\end{proof}

The above lemma implies that $\ext ^1_{\PP^3}(E,E)$ is
$4c$-dimensional. Since $\hom _R(X,X)$ is $c$-dimensional, the
orbit $\O (X)$ is of dimension $c^2-c$. But then the dimension of
${\Mod} _R^c$ at $X$ is at most $c^2-c+4c=c^2+3c$. Since the
pre-image of the closed subscheme $\varphi (S^c\AA ^4)$ in ${\Mod}
_R^c$ is of dimension at least $4c +(c^2-c)$ (as all the fibers of
the restricted map contain closed orbits of dimension  at least
$c^2-c$), we see that $\varphi (S^c\AA ^4)$ is an irreducible
component of $ \M _0 (\PP ^3;0,c)$.
\end{proof}

\medskip
\begin{Remark}
One can easily see that $\dim \ext ^2_{\PP^3}(\O_l, \O_l)=3.$
Therefore the instanton $E$ from the above proof has $\dim \ext
^2_{\PP^3}(E,E)=3c$  so it is potentially obstructed (cf. Theorem
\ref{deformation}). On the other hand, the above proof shows that
the corresponding point in $ \M _0 (\PP ^3;0,c)$ is smooth.
\end{Remark}

\medskip

The following example shows that $\M _0 (\PP ^3;0,c)$ need not be
irreducible (unlike in the case of ADHM data for a point). But it
is still possible that it is a connected locally complete
intersection of dimension $4c$.
\medskip

\begin{Example}
Let us consider ADHM data on $\PP ^1$ for $r=0$ and $c=2$ in the
characteristic zero case.  In this case one can compute that $\ti
\mu ^{-1}(0)$ has two irreducible and reduced components: $X_1$ of
dimension $11$ and $X_2$ of dimension $10$ intersecting along an
irreducible and reduced scheme of dimension $9$ (to see this fact
we first performed a computer assisted computation in {\tt
Singular}). We can explicitly describe these two components as
follows.

Let $V_1$ and $V_2$ denote varieties of pairs of commuting
$2\times 2$ matrices (see Example \ref{pairs-of-matrices}). Let us
note that ${\ti \mu}^{-1}(0)$ is a subvariety in $V_1\times V_2$
given by equation $[B_{11},B_{22}]+[B_{12},B_{21}]=0$, where
$(B_{11}, B_{21})\in V_1$ and $ (B_{12}, B_{22})\in V_2$ are pairs
of $2\times 2$ matrices.

Let us set
$$B_{1k}=\left(\begin{array}{cc}
1&0\\
y_{1k}&1
\end{array}\right)
\left(\begin{array}{cc}
y_{3k}&y_{2k}(y_{3k}-y_{4k})\\
0&y_{4k}
\end{array}\right)
\left(\begin{array}{cc}
1&0\\
-y_{1k}&1
\end{array}\right),
$$
and
$$B_{2k}=\left(\begin{array}{cc}
1&0\\
y_{1k}&1
\end{array}\right)
\left(\begin{array}{cc}
y_{5k}&y_{2k}(y_{5k}-y_{6k})\\
0&y_{6k}
\end{array}\right)
\left(\begin{array}{cc}
1&0\\
-y_{1k}&1
\end{array}\right).
$$
As in Example \ref{pairs-of-matrices} the condition
$[B_{1k},B_{2k}]=0$ is satisfied for both $k=1$ and $k=2$.

Thus we can define a map $\psi$ from $\AA^{12}$ to the product
$V_1\times V_2$ by sending $(y_{ij})$ to $(B_{11}, B_{21}, B_{12},
B_{22})$ defined above. Computations in {\tt Singular} show that
$\psi^{-1}(\ti \mu^{-1}(0))$ has three irreducible components
$Y_1, Y_2, Y_3$ given by the following ideals:
$$I_1=(
(y_{31}-y_{41})(y_{52}-y_{62})-(y_{51}-y_{61})(y_{32}-y_{42})),$$
$$I_2=(y_{11}-y_{12},y_{21}-y_{22})$$
and
$$I_3=(y_{21}y_{12}-y_{21}y_{11}+1,y_{21}+y_{22}).$$

Further computations show that $\psi(Y_1)$ is $11$-dimensional and
$\psi(Y_2)$ and $\psi(Y_3)$ are equal and $10$-dimensional. This
shows that the restriction $\psi_{|Y_1}$ is a generically finite
morphism from $Y_1$ to the $11$-dimensional component $X_1$ of
$\mu^{-1}(0)$. Similarly, $\psi |_{Y_2}$ and  $\psi |_{Y_3}$ are
generically finite morphisms from $Y_2$ and $Y_3$ to the
$10$-dimensional component $X_2$.

We have dominant morphisms $Y_1\to X_1/\GL (2)$ and $Y_2\to
X_2/\GL (2)$. For a quadruple of matrices $(B_{11}, B_{21},
B_{12}, B_{22})\in X_2$ obtained as the image of a point
$(y_{ij})\in Y_2$, the isotropy group of $\GL (2)$ contains
matrices of the form
$$\left(\begin{array}{cc}
1&0\\
y_{11}&1
\end{array}\right)
\left(\begin{array}{cc}
t_1&y_{21}(t_1-t_2)\\
0&t_2
\end{array}\right)
\left(\begin{array}{cc}
1&0\\
-y_{11}&1
\end{array}\right)
$$
for arbitrary $t_1,t_2\in \GG _m$. One can see that  $Y_2$ is
mapped dominantly onto the image of $S^2\AA ^4$ in ${\ti
\mu}^{-1}(0)/\GL (2)$ and therefore for a generic quadruple
$(B_{11}, B_{21}, B_{12}, B_{22})\in X_2$ the isotropy group is
$2$-dimensional and it is equal to the above described group (one
can also compute this isotropy group explicitly for all such
quadruples).

One can also check that the isotropy group of a generic point in
$X_1$ is $1$-dimensional (so the corresponding $R$-module is
simple) and therefore ${\ti \mu}^{-1}(0)/\GL (2)$ is pure of
dimension $8$ with irreducible components given by $X_1/\GL (2)$
and $X_2/\GL (2)$.

This also proves that the injection $\varphi: S^2\AA^4\to \ti \mu
^{-1}(0)/\GL (2)$ maps $S^2\AA^4$ onto an irreducible component of
the quotient ${\ti \mu} ^{-1}(0)/\GL (2)$.

Now we need to check that the components $X_1/\GL (2)$ and
$X_2/\GL (2)$ do not coincide. For this we need the following
lemma:

\begin{Lemma}
Let $G$ be a linear algebraic group acting on a reducible variety
$X$ with two irreducible components $X_1$ and $X_2$. Then $X_1/G
\cap X_2/G=(X_1\cap X_2)/G$.
\end{Lemma}
\begin{proof}
First observe that since $G$ is irreducible, the closure of an
orbit of a point $x\in X$ is contained in the same irreducible
component as $x$. The intersection $X_1\cap X_2$ is closed and
$G$-invariant so $(X_1\cap X_2)/G$ can be regarded as subvariety
in $X_1/G \cap X_2/G$.

Let us take $y\in X_1/G \cap X_2/G\subset X/G$. By assumption it
is the image of a closed orbit of a point $x\in X$. We claim that
$x\in X_1\cap X_2$. Note that $y$ is the image of some points
$x_1\in X_1,x_2\in X_2$. Since each $\overline{Gx_i}\subset X_i$
for $i=1,2$ contains a unique closed orbit, it must be the orbit
of $x$ and it is contained in $X_1\cap X_2$. Therefore $y$ lies in
$(X_1\cap X_2)/G$.
\end{proof}

The image of a quadruple  $(B_{11}, B_{21}, B_{12}, B_{22})\in
X_2$ in the image of $S^2\AA^4$ in ${\ti \mu}^{-1}(0)/\GL (2)$ is
given by quadruples of pairs of eigenvalues of matrices $B_{ij}$.
But for a quadruple of $2\times 2$ matrices $(B_{11}, B_{21},
B_{12}, B_{22})\in X_1\cap X_2$ obtained as the image of
$(y_{ij})\in Y_1\cap Y_2$ we have the equation
$$(y_{31}-y_{41})(y_{52}-y_{62})=(y_{51}-y_{61})(y_{32}-y_{42})$$
for the eigenvalues. Therefore the image of $Y_1\cap Y_2$ in ${\ti
\mu}^{-1}(0)/\GL (2)$ has dimension $7$. Together with the above
lemma this proves the following corollary:

\begin{Corollary}
$\M _0 (\PP ^3;0,2)$ has two $8$-dimensional irreducible
components intersecting along a $7$-dimensional variety.
\end{Corollary}

\end{Example}

\section{Examples and counterexamples}

In this section we consider generalized ADHM data in the case
$X=\PP ^1$. We provide a few examples showing, e.g., a relation
between our notion of stability and that of Frenkel and Jardim. We
also show a few counterexamples to some expectations of Frnekel
and Jardim.

In this section we keep notation from Section \ref{FJ-section}.

\subsection{Relation between GIT semistability and FJ-semistability}

The following lemma follows immediately from definitions:

\begin{Lemma}
Let us fix an ADHM datum $x\in \ti{\gr{B}}=\gr{B}\otimes
H^0(\O_{\PP ^1}(1))$. If $x$ is FJ-semistable then it is also
stable.
\end{Lemma}

Jardim in \cite[Proposition 4]{J} claims that the opposite
implication also holds but the following example shows that this
assertion is false.

\medskip

\begin{Example} \label{GITvsFJ}
We consider ADHM data in case $r=1$ and $c=2$. Let us fix
coordinate systems in $V$ and $W$ and consider an element
$x=(\fal{B_1},\fal{B_2},\fal{i},\fal{j}) \in\fal{\gr{B}}$ given by
$$\tilde{B}_1=\left[\begin{array}{cc}x_0&x_0\\x_1&x_1\end{array}\right],
\quad \tilde{B}_2
=\left[\begin{array}{cc}x_0&-x_0\\x_1&-x_1\end{array}\right],
\quad \ti{i} =\left[\begin{array}{c}x_0\\x_1\end{array}\right]
\quad \hbox{and} \quad \ti{j}
=\left[\begin{array}{cc}-2x_1&2x_0\end{array}\right].
$$
It is easy to see that $\ti \mu (x)=0$. Hence $x$ is an ADHM
datum.

We claim that this ADHM datum is stable and costable. To prove
that consider a vector subspace $S\subset V$ such that
$\im\fal{i}\subset S\otimes\hox$. We claim that $S$ must be
two-dimensional. Otherwise, there exist constants $a, b \in k$
such that every element in $S\otimes\hox$ can be written as
$\left[\begin{array}{c}a f(x_0,x_1)\\b f(x_0,x_1)\end{array}
\right]$ for some linear polynomial $f$ in $x_0$ and $x_1$. But
$\left[\begin{array}{c}x_0\\x_1\end{array}\right]\in\im\fal{i}$
cannot be written in this way. Therefore $S=V$, which proves that
the ADHM datum $x$ is stable. Since $\ker\fal{j}=0$, the ADHM
datum $x$ is also costable.

Now fix a point $p=[a:b]\in\PP^1$ and consider the subspace
$S\subset V=k^2$ spanned by  vector $s=
\left[\begin{array}{c}a\\b\end{array}\right] $. Then
$$(\fal{B}_1 (p))(s)
=(a+b)\cdot  s , \quad (\fal{B}_2 (p))(s) =(a-b)\cdot  s, \quad
(\fal{i} (p))(1) =s \quad \hbox{and} \quad (\fal{j} (p)) (s)=0 .$$
Therefore $(x(p)) (S)\subset S$ and  $x$ restricted to any point
of $\PP^1$ is neither stable nor costable. In particular, the ADHM
datum $x$ is regular but not FJ-semistable.
\end{Example}

\medskip

Let us focus on the example given above and study cohomology
groups of the complex $\C_x^{\bullet}$ corresponding to the ADHM
datum $x$.

First let us  describe the locus of points $p=[x_0,x_1,x_2,x_3]\in
\PP ^3$ where the map $\alpha (p)$ is not injective. It is
equivalent to describing the locus
$$\rk\left(\begin{array}{cc}
x_0+x_2&x_0\\
x_1&x_1+x_2\\
x_0+x_3&-x_0\\
x_1&-x_1+x_3\\
-2x_1&2x_0\\
\end{array}
\right)\le 1.$$ Easy computations show that this set is an
intersection of two planes:
$$
\left\{\begin{array}{l}
x_0+x_1+x_2=0\\
x_0-x_1+x_3=0
\end{array}
\right. .
$$
Similarly, the locus of points $p \in \PP ^3$ where
$$\beta ( p) =\left(\begin{array}{ccccc}
-x_0-x_3&x_0&x_0+x_2&x_0&x_0\\
-x_1&x_1-x_3&x_1&x_1+x_2&x_1
\end{array}\right)$$ is not surjective is the line given by equations $x_2=x_3=0$.

Using this one can see that $\H ^1(\C_x^{\bullet})$ is a pure
sheaf of dimension $1$ and $\H ^0(\C_x^{\bullet})$ is a torsion
free sheaf whose reflexivization is locally free.

\subsection{Relation to Diaconescu's approach to ADHM data}

We will use notation from \cite[Section 2]{Di} (see also
\cite[2.9.2]{Sch}). Let us set $\X=(X, M_1=\O _X (-1),
M_2=\O_X(-1), E_{\infty}=W\otimes \O_X(-1))$ and consider an ADHM
sheaf $\E=(E=V\otimes \O_X, \Phi _1, \Phi_2, \varphi, \psi)$ for
this data.

\begin{Definition}
We say that $\E$ is \emph{stable} if for every subspace
$S\subsetneq V$ (possibly $S=0$) such that $\Phi_k(S\otimes
\O_X(-1))\subset S\otimes \O_X$ for $k=1,2$ we have $\im \Psi
\not\subset S\otimes \O_X$.
\end{Definition}

The above stability notion is similar to Diaconescu's stability
\cite[Definition 2.2]{Di} but with stability condition only for
subsheaves $E'$ of the form $S\otimes \O_X$ for some $0\subsetneq
S\subsetneq V$.

Let $\ti B_k: V\to V\otimes H^0(\O_X(1))$ be induced by $\Phi_k$
and let $\ti i:W\to V\otimes H^0(\O_X(1)) $ and $\ti j: V\to
W\otimes H^0(\O_X(1))$ be induced by $\psi$ and $\varphi$,
respectively.

Then giving $\E$ is equivalent to giving a point $x=(\ti B_1, \ti
B_2, \ti i, \ti j )\in \ti \bB$ such that $\ti \mu (x)=0$.
Moreover, $\E$ is stable in the above sense if and only if $x$ is
stable.

\medskip

\subsection{Counterexample to the Frenkel--Jardim conjecture} \label{counter}

In \cite{FJ} Frenkel and Jardim conjectured that the moduli space
$ \M ^{f}(\PP ^3; r,c)$ of framed instantons is smooth and
irreducible. Here we show that this conjecture is false.

Let us consider the map $\varphi: \O_{\PP^3}^4\to \O_{\PP^3}^3(1)$
given by
$$\varphi=
\left(\begin{array}{cccc}
x_2&x_3&0&0\\
0&x_2&x_3&0\\
0&0&x_2&x_3\\
\end{array}\right) .
$$
Now let us consider the sheaf $E$ defined by the short exact
sequence
$$0\to E\longrightarrow
\O_{\PP^3}^4\mathop{\longrightarrow}^{\ti \varphi } \O_{m}(1)^3\to
0,$$ where $m$ is the line $x_0=x_1=0$ and $\ti \varphi$ is the
composition  the natural restriction map $\O_{\PP^3}(1)^3\to
\O_{m}^3(1)$ with $\varphi$. It is easy to see that $E$ is a
$(4,3)$-instanton, trivial on the line $l_{\infty}:=(x_2=x_3=0)$.
From the defining sequence we have an exact sequence
$${\ext}^2(\O_{\PP^3}^4, E)\to {\ext}^2(E, E) \to {\ext}^3(\O_{m}(1)^3, E)\to {\ext}^3(\O_{\PP^3}^4, E).$$
Then ${\ext}^l(\O_{\PP^3}^4, E)=H^l(E)^4=0$ for $l=2,3$ and
${\ext}^3(\O_{m}(1)^3, E)$ is Serre dual to $\hom(E,
\O_{m}(-3)^3)$. But after restricting to $m$ we have
$$E|_m\twoheadrightarrow \O_m(-3)\to \O_m^4\to \O_{m}(1)^3\to 0,$$ and
it is easy to see that $\hom(E, \O_{m}(-3))$ is $1$-dimensional.
In particular, $\dim \ext^2(E,E)=3$.

On the other hand, by Lemma \ref{existence} there exists a locally
free $(4,3)$-instanton $F$ such that $\ext ^2(F,F)=0$. It
corresponds to a smooth point of an irreducible component of
expected dimension (see Theorem \ref{deformation}). Therefore the
point corresponding to $E$ in the moduli space of framed
instantons is either singular or lives in a component of
unexpected dimension (in which case the moduli space would not be
irreducible).

Another way of looking at this example is defining an ADHM datum
for $\PP ^1$, for $r=4$, $c=3$. We define an ADHM datum $x=(\ti
B_1, \ti B_2, \ti i , \ti j)$ by setting $\ti B_1=0, \ti B_2=0,
\ti j=0$ and
$$\ti i=\left(\begin{array}{cccc}
x_0&x_1&0&0\\
0&x_0&x_1&0\\
0&0&x_0&x_1\\
\end{array}\right).$$
It is easy to see that these matrices satisfy the ADHM equations
and define an FJ-stable ADHM datum (in fact, $\ti i _p$ is
surjective for every $p\in \PP ^1$). The corresponding
torsion-free framed $(4,3)$-instanton $E$ can be described by the
above sequence. In terms of ADHM data we proved that the moment
map $\mu$ is not submersion at $x$ (see Theorem \ref{deformation})
but there exist ADHM data at which $\mu$ is a submersion.

More generally, one can easily see that if $c<r<3c/2$ then $M
(l_{\infty}; r,c)$ is either singular or reducible.  Indeed, one
can find an FJ-stable complex ADHM data $x\in \BB$ for which only
$\ti i$ is non-zero. Then the rank of $d\mu_x$ is at most
$2cr<3c^2$, so $d\mu_x$ is not surjective. On the other hand, by
Lemma \ref{existence} there exists an irreducible component of
expected dimension which proves our claim.

\medskip

\subsection{Weak instantons} \label{coanda}

\begin{Definition}
A \emph{weakly instanton sheaf} (or a \emph{weak instanton}) is a
torsion free sheaf $E$ on $\PP ^3$ such that
\begin{itemize}
\item $c_1(E)=0$,
\item $H^0(E(-1))=H^1(E(-2))=H^3(E(-3))=0$.
\end{itemize}
\end{Definition}

Weak instantons were introduced by Frenkel and Jardim (see
\cite[2.4]{FJ}) to deal with FJ-semistable data in the rank $1$
case.

\medskip

We say that a torsion free sheaf on $\PP^3$ has \emph{trivial
splitting type} if there exists a line such that the restriction
of this sheaf to a line is a trivial sheaf. In this case the
restriction to a general line is also a trivial sheaf.

\begin{Lemma}
Let $E$ be a locally free sheaf on $\PP^3$ of trivial splitting
type. Then $H^0(E(-1))=H^3(E(-3))=0$. In particular, if
$H^1(E(-2))=0$ then $E$ is a weak instanton.
\end{Lemma}

\begin{proof}
If $E$ is of trivial splitting type then both $E(-1)$ and
$E^*(-1)$ have no sections. Since $H^3(E(-3))$ is Serre dual to
$H^0(E^*(-1))$ this shows the first part. The second one follows
from the first one by noting that for a sheaf of trivial splitting
type we have $c_1(E)=0$.
\end{proof}

\begin{Definition}
We say that a perverse instanton  $\C$ is \emph{mini-perverse} if
$\H ^0(\C)$ is torsion free and $\H^1(\C)$ is a sheaf of finite
length.
\end{Definition}

Obviously, any instanton is  mini-perverse, but the opposite
implication does not hold.

\medskip
\begin{Lemma} \label{mini-perv}
An ADHM datum $x\in \ti {\bf B}$ is FJ-semistable if and only if
the corresponding perverse instanton $\C ^{\bullet}_x$ is
mini-perverse.
\end{Lemma}

\begin{proof}
The ``if'' implication is a content of \cite[Proposition 17]{FJ}.
To prove the converse note that the restriction of a mini-perverse
instanton $\C_A$ corresponding to $A\in \ti {\bf B}$ to a general
hyperplane containing $l_{\infty}$ gives a locally free sheaf on
$\PP^2$. But this shows that for a general point $x\in \PP ^1$ the
ADHM datum $A(x)$ (corresponding to this restriction) is regular.
\end{proof}

\medskip

Note that in Example \ref{GITvsFJ} the constructed perverse
instanton is not mini-perverse. So the above lemma gives another
proof that this perverse instanton is not FJ-semistable.
\medskip

\begin{Lemma}
If $\C$ is a mini-perverse instanton then $\H^0(\C)$ is a weak
instanton of trivial splitting type.
\end{Lemma}

\begin{proof}
Let us set $E=\H ^0(\C)$ and $T=\H^1(\C)$. If $\C$ is a perverse
instanton then we have the distinguished triangle $$E\to \C\to
T[-1]\to E[1].$$ The long cohomology exact sequence for this
triangle gives exactness of the following sequence:
$$ 0=H^0(\C (-2))\to H^1(T(-2))\to H^1(E(-2))\to H^1(\C (-2))=0.$$
Since $T$ has dimension zero we see that $H^1(E(-2))=0$ and so $E$
is a weak instanton. The fact that it is of trivial splitting type
follows from the fact that $\H ^0(\C)$ is trivial on $l_{\infty}$.
\end{proof}

\medskip

\begin{Lemma}\label{zero-dim}
A zero dimensional coherent sheaf $E$ on a smooth variety $X$ has
homological dimension equal to the dimension of $X$.
\end{Lemma}

\begin{proof}
By the Auslander--Buchsbaum theorem it is sufficient to prove that
$E_x=E\otimes \O _{X,x}$ has depth zero. Assume that it has depth
at least $1$. Then there exists an element $y\in m_x\subset
\O_{X,x}$ such that multiplication by $y$ defines an injective
homomorphism $\varphi _y : E_x\to E_x$. Note that $\varphi _y$ is
an isomorphism since $E_x$ is zero dimensional and $H^0(\varphi
_y)$ is an isomorphism as it is a linear injection of $k$-vector
spaces of the same dimension. But this implies that $m_xE_x=E_x$
which contradicts Nakayama's lemma.
\end{proof}

\medskip

\begin{Lemma} \label{locfree-mini-perv}
If a locally free sheaf $E$ appears as $\H^0(\C)$ for some
mini-perverse instanton $\C$ then $\H ^1(\C)=0$. In particular,
$E$ is an instanton.
\end{Lemma}

\begin{proof} By Lemma \ref{perv-family} $\C$ is isomorphic in
$D^b(\PP^3)$ to the complex
$$(0\to \O_{\PP^3}(-1)^c\mathop{\longrightarrow}^{\alpha}
\O_{\PP^3}^{2c+r}\mathop{\longrightarrow}^{\beta}
\O_{\PP^3}(1)^c\to 0).$$ Set $T=\H^1(\C)$. We have a short exact
sequence
$$0\to\O_{\PP^3}(-1)^c{\simeq}\im \alpha\to \ker \beta \to E\to 0$$
which, together with our assumption on $E$, implies that $\ker
\beta$ is locally free. On the other hand, we have an exact
sequence
$$0\to \ker \beta \to \O_{\PP^3}^{2c+r}  \to \O_{\PP^3}(1)^c \to T\to 0,$$
which implies that the homological dimension of $T$ is at most
two.

But  if $T\ne 0$ then Lemma \ref{zero-dim} implies that the
homological dimension of $T$ is equal to $3$, a contradiction.
\end{proof}
\medskip

\begin{Example}
In  \cite[2.4]{FJ} Frenkel and Jardim ask if every weak instanton
of trivial splitting type come from some FJ-semistable ADHM datum.
In view of Lemma \ref{mini-perv} this would imply that such an
instanton is of the form $\H ^0(\C) $ for some mini-perverse
instanton $\C$. Here we give a negative answer to this question.
Note that if the answer were positive then by Lemma
\ref{locfree-mini-perv} every locally free weak instanton of
trivial splitting type would be an instanton. So it is sufficient
to show a weakly instanton sheaf which is locally free of trivial
splitting type but which is not an instanton.

We use \cite[Example 1.6]{Co} to show a rank $3$ locally free
sheaf $E$ on $\PP^3$ which is trivial on a general line and has
vanishing $H^1(E(-2))$ and it does not appear as $\H^0(\C)$ for
some mini-perverse instanton (there are no such sheaves in the
rank $2$ case).  This gives a negative answer to the question
posed in \cite[2.4]{FJ}.

Let $q\ge 1$ and $c_2\ge 2q$ be integers. Let $Z_1$ and $Z_2$ be
plane curves of degree $c_2-q$ and $q$ contained in different
planes. Assume that they intersect in $0\le s\le q$ simple points
and set $Z=Z_1\cup Z_2$. Then there exists a rank $3$ vector
bundle $E$ which sits in a short exact sequence
$$0\to \O_{\PP^3}^2\to E\to I_{Z}\to 0.$$ Then $E$ is trivial along
any line disjoint with $Z$.

Using the short exact sequence
$$0\to I_Z\to \O_{\PP ^3}\to \O _Z\to 0$$
we see that $H^1(I_Z(-2))=0$. Therefore $H^1(E(-2))=0$. Obviously,
$H^0(E(-1))=0$. Since $E$ is locally free, the Serre duality
implies that $H^3(E(-3))$ is dual to $H^0(E^* (-1))$. But $E$ is
slope semistable and hence $E^*(-1)$ has no sections. Thus $E$ is
a weak instanton of trivial splitting type.

On the other hand, $H^2(E(-2))$ has dimension $\chi (E(-2))$ as
all the other cohomology of $E(-2)$ vanish. But the Riemann--Roch
theorem implies that  $\chi (E(-2))={1\over 2}c_3=s+q^2+{1\over
2}c_2(c_2-2q+1)$, so $E$ is not an instanton.
\end{Example}

\subsection{Perverse instantons of charge $1$}

In \cite{FJ} the moduli spaces of framed torsion free instantons
with $c=1$ and $r\ge 2$ were described quite explicitly. Let us
recall that such instantons come from FJ-stable ADHM datum.  We
can generalize this description to the case of stable ADHM datum.
For $c=1$ general ADHM datum consists of complex numbers $B_{lk}$
and $i_k$, $j_k$ which can be regarded as vectors in $W$. The ADHM
equation reduce to
\begin{equation}
\label{reduced_ADHM_eq} \fal{i}\fal{j}=0.
\end{equation}
Stability is equivalent to $\fal{i}\neq 0$ and costability to
$\fal{j}\neq 0$. The group $\GL(V)$ is just $\Gm$ and $t\in\Gm$
acts trivially on $\fal{B}_k$, it acts on $i_k$ by multiplication
by $t$ and on $j_k$ by multiplication by $t^{-1}$. The moduli of
perverse instantons for fixed $r\ge 1$ and $c=1$ is isomorphic to
$\AA^4\times \B(r)$ where $\B(r)$ is the set of solutions of
equation (\ref{reduced_ADHM_eq}) modulo the action of $\Gm$. Note
however, that there exist stable ADHM data also for $r=1$ whereas
there are no FJ-stable ones (see \cite[Propositions 4 and
15]{FJ}).

\begin{Proposition}
For $r\ge 2$ $\B(r)$ is a quasi projective variety of dimension
$4(r-1)$ and $\B (1) \simeq\PP^1$.
\end{Proposition}

\begin{proof}
We follow the proof of Proposition 7 in \cite{FJ}. Let
$$i_1=(x_1,\ldots,x_r), \quad i_2=(y_1,\ldots,y_r),$$
$$j_1=\left(\begin{array}{c}
 z_1 \\ \vdots \\ z_r
 \end{array}\right),\quad j_2\left(\begin{array}{c}
 w_1 \\ \vdots \\ w_r
 \end{array}\right).$$
Then equation (\ref{reduced_ADHM_eq})  reduces to
\begin{equation}
\sum_{k=1}^r x_kz_k=\sum_{k=1}^ry_kw_k=\sum_{k=1}^rx_kw_k+y_kz_k=0.\label{ADHM_equations}
\end{equation}
Such an ADHM datum is stable if and only if $i_1$ or $i_2$ is not
a zero vector. One can also easily show that FJ-stability is
equivalent to the vectors $i_1$ and $i_2$ being linearly
independent. $\B(r)$ is the complete intersection of the three
quadrics (\ref{ADHM_equations}) in the open subset of the
$(4r-1)$-dimensional weighted projective space
$$X=\PP(\underbrace{1,\ldots,1}_{2r},\underbrace{-1,\ldots,-1}_{2r}).$$
This shows that $\B(r)$ is quasi-projective.

\begin{Remark}
A point in the complete intersection of the quadrics
(\ref{ADHM_equations}) in $X$ corresponds to an ADHM datum which
is either stable or costable.
\end{Remark}

Let us consider the map
$$\mu:\AA^{4r}\rightarrow \AA^3$$
given by
\begin{eqnarray*}
\mu(x_1,\ldots,x_r,y_1,\ldots,y_r,z_1,\ldots,z_r,w_1,\ldots,w_r)=
\left(\sum_{k=1}^r
x_kz_k,\sum_{k=1}^ry_kw_k,\sum_{k=1}^rx_kw_k+y_kz_k\right).
\end{eqnarray*}
The derivative of $\mu$ is given by
$$
D\mu=\left(
\begin{array}{cccccccccccc}
z_1&\ldots&z_r&0&\ldots&0&x_1&\ldots&x_r&0&\ldots&0\\
0&\ldots&0&w_1&\ldots&w_r&0&\ldots&0&y_1&\ldots&y_r\\
w_1&\ldots&w_r&z_1&\ldots&z_r&y_1&\ldots&y_r&x_1&\ldots&x_r
\end{array}
\right)
$$
Frankel and Jardim claimed that for $r\ge 2$ the matrix $D\mu$ has
maximal rank $3$ if and only if $(x_1,\ldots,x_r)$ and
$(y_1,\ldots,y_r)$ are linearly independent. However, only the
implication "$\Leftarrow$" is true and their result on
non-singularity at points corresponding to FJ-stable ADHM data
remains correct. It also follows that $\dim\B(r)=4r-4$. In
characteristic different from $2$, setting $x_1=z_2=w_2=1$,
$y_1=2$ and all other coefficients equal $0$ gives an example of
stable ADHM datum which is not FJ-stable but it corresponds to a
nonsingular point in the moduli space of perverse instantons. On
the other hand, if $i_1$ and $i_2$ are linearly dependent and
$j_1=j_2=0$ then $D\mu$ has clearly rank $2$. This shows a stable
ADHM datum which is neither costable nor FJ-stable but it gives a
singular point.

In the case $r=1$, equations (\ref{ADHM_equations}) reduce to
$\fal{i}=0$ or $\fal{j}=0$. Stability is equivalent to
$\fal{i}\neq 0$ so $D\mu$ has rank $2$ for all stable ADHM datum.
Clearly, we have $\B(1)\simeq \PP^1$.

\end{proof}

\section{A general study of ADHM data for $\PP ^1$}

In this section we introduce a hypersymplectic reduction which is
a holomorphic analogue of a hyper-K\"ahler structure. We also
relate the moduli space of framed instantons to the moduli space
of framed modules of Huybrechts and Lehn. The relation is not as
straightforward as in the surface case since many framed
instantons are not Gieseker $\delta$-semistable framed modules on
$\PP ^3$ for all parameters $\delta$. The relation shows existence
of the moduli space of framed instantons without Theorem
\ref{perv-moduli}.

\subsection{Hypersymplectic reduction}

Let $X$ be a a smooth quasi-projective $k$-variety. As an analogue
of a hyper-K\"ahler structure we introduce the following:

\begin{Definition}
We say that $X$ has a \emph{hypersymplectic structure} if there
exist a non-degenerate symmetric form $g$ on $TX$ and maps of
vector bundles $I,J,K: TX\to TX$ such that
\begin{enumerate}
\item $g(Iv, Iw)=g(Jv, Jw)=g(Kv, Kw)=g(v,w),$
\item $I^2=J^2=K^2=IJK=-1.$
\end{enumerate}
\end{Definition}

If we have a hypersymplectic manifold then we can define
non-degenerate symplectic forms $\omega_1, \omega_2, \omega_3$ on
$X$ by $\omega_1(v,w)=g(Iv,w)$, $\omega_2(v,w)=g(Jv,w)$ and
$\omega_3(v,w)=g(Kv,w)$.  Assume that there exists a reductive
$k$-group $G$ acting on $X$ and preserving $g,I,J,K$. As an
analogue of a hyper-K\"ahler moment map we have the following:

\begin{Definition}
A map $\mu=(\mu_1 , \mu_2 , \mu_3): X\to k^3\otimes \fg ^*$ is
called a \emph{hypersymplectic moment map} if it satisfies the
following properties:
\begin{enumerate}
\item $\mu_l$ is $G$-equivariant for $l=1,2,3$,
\item $\langle d\mu_{l,x}(v), \xi\rangle =\omega _l (\xi_x, v)$ for $l=1,2,3$ and
for any $x\in X$, $v\in T_xX$ and $\xi \in \fg$.
\end{enumerate}
\end{Definition}

Let $V_x$ be the image of the tangent map (at the unit) to the
orbit map $\varphi_x: G\to X$ sending $g$ to $g\cdot x$.

\begin{Proposition} \label{hyper-sympl}
Let us take a point $x\in X$.  Then the following conditions are
equivalent:
\begin{enumerate}
\item $d\mu_x$ is surjective.
\item $d \varphi _x$ is an injection and $S=IV_x+JV_x+KV_x$ is a direct sum.
\item The map $\fg\oplus \fg\oplus \fg \to T_xX$ given by
$(\xi_1,\xi_2,\xi_3)\to (I\xi_{1,x},J\xi_{2,x},K\xi_{3,x})$ is
injective.
\end{enumerate}
\end{Proposition}

\begin{proof}   Since
$$\langle d\mu_x(v), \xi \rangle=(g(I\xi_x, v),g(J\xi_x,
v),g(K\xi_x, v)),$$ the kernel of $d\mu_x$ is equal to the
orthogonal complement $S^\perp$ of $S$ (with respect to $g$).
Since $g$ is non-degenerate we have
$$\dim S+\dim S^{\perp}=\dim X.$$ Hence $d\mu_x$ is surjective if and
only if $\dim S=3\dim \fg$.  This is clearly equivalent to saying
that $d\varphi_x$ is injective (i.e., $\dim V_x=\dim \fg$) and
$IV_x+JV_x+KV_x$ is a direct sum. Equivalence with the last
condition is clear.
\end{proof}

\medskip

\begin{Proposition} \label{hyper2}
Let $\eta =(\eta_1, \eta_2, \eta_3) \in \fg^*\oplus \fg^*\oplus
\fg^*$ satisfy $\Ad ^*_{g}(\eta_i)=\eta_i$ for all $g\in G$. If
$x\in \mu^{-1} (\eta)$ and $g|_{V_x}$ is non-degenerate then
$d\mu_x$ is surjective.
\end{Proposition}

\begin{proof}
By assumption $\mu$ sends a $G$-orbit of $x$ into a point. Hence
$d\mu_x(V_x)=0$. This immediately implies that
$V_x,IV_x,JV_x,KV_x$ are orthogonal to each other (with respect to
$g$). But then the assertion follows from the above proposition.
Indeed, if there exists $(\xi_1,\xi_2,\xi_3)$ such that
$I\xi_{1,x}+J\xi_{2,x}+K\xi_{3,x}=0$ then $g(\xi_{1,x}, \zeta_x)=
g(I\xi_{1,x}+J\xi_{2,x}+K\xi_{3,x}, I\zeta_{x})=0$ for any $\zeta
\in \fg$. Therefore $\xi_{1,x}=0$ and similarly
$\xi_{2,x}=\xi_{3,x}=0$.
\end{proof}

\medskip

Note that it can easily happen that the form $g$ restricted to
$V_x+S$ is zero and $d\mu_x(S)=0$ although $d\mu_x$ is surjective
(this happens, e.g., in Example \ref{counter}).

\subsection{Hypersymplectic moment map for ADHM data on $\PP^1$.}

In this subsection we assume that  the characteristic of the base
field is zero.

Let us fix a basis $x_0,x_1$ of $H^0(\PP ^1, \O _{\PP ^1} (1))$.
Then a point $x\in \ti {\bf B}={\bf B}\otimes H^0(\PP ^1, \O _{\PP
^1}(1))$ can be thought of as a matrix
$$x=\left(\begin{array}{cccc}
B_{11} & B_{12} & i_1 & j_1\\
B_{21} & B_{22} & i_2 & j_2\\
\end{array}\right) ,
$$
where $(B_{l1}, B_{l2}, i_l, j_l)\in \bf B$ for $l=1,2$ is written
as in case of the usual ADHM data. Using this notation we define a
symmetric form $g$ on $T \ti {\bf B}$ by
$$g(x,x')=
\Tr (B_{11}B_{22}'+B_{22}B_{11}'-B_{21}B_{12}'-B_{12}B_{21}'+
i_1j_2'+i_1'j_2-i_2j_1'-i_2'j_1).$$ Let us choose a standard
quaternion basis:
$$I=\left(\begin{array}{cc}
\sqrt{-1}&0\\
0&-\sqrt{-1}\\
\end{array}\right) , \quad \quad J=\left(\begin{array}{cc}
0&1\\
{-1}&0\\
\end{array}\right) ,  \quad \quad K=\left(\begin{array}{cc}
0&\sqrt{-1}\\
\sqrt{-1}&0\\
\end{array}\right).
$$
Then $I,J,K$ can be thought of as operators acting on $T \ti {\bf
B}$. Let us write $\ti \mu :\ti {\bf B}\otimes H^0(\PP ^1, \O
_{\PP ^1} (1)) \to \enD V \otimes H^0(\PP ^1, \O _{\PP ^1}(2)) $
as the sum $\mu_1x_0^2+\mu_2x_0x_1+\mu_3x_1^2$ in which $\mu _l:
\ti {\bf B}\to \enD V$ for $l=1,2,3$ are the corresponding
components. Let us set $\ti\mu_1(x) = \sqrt{-1} \mu_2(x)$,
$\ti\mu_2(x)=\mu_1(x)+\mu_3(x)$ and
$\ti\mu_3(x)=\sqrt{-1}(-\mu_1(x)+\mu_3(x))$.

By a straightforward computation we get the following proposition:

\begin{Proposition}
$(g,I,J,K)$ define a hypersymplectic structure on $\ti{\bf B}$.
Moreover, $\ti\mu=(\ti\mu_1 , \ti\mu_2 , \ti\mu_3): \ti{\bf B} \to
k^3\otimes \fg ^*$ is a hypersymplectic moment map.
\end{Proposition}

This, together with Proposition \ref{hyper-sympl}, implies the
following corollary which can be used for checking smoothness of
the moduli space of framed perverse instantons:

\begin{Corollary}
Let $x\in \ti{\bf B}$ be a stable ADHM datum. Then $d\ti\mu_x$ is
surjective if and only if there exist no $(\xi _{1},\xi_2,\xi_3
)\in \fg\oplus \fg\oplus\fg -\{ (0,0,0) \}$ such that
$$\xi_{1,x}+I\xi_{2,x}+J\xi_{3,x}=0. $$
\end{Corollary}

\subsection{Relation to moduli spaces of framed modules}

Let $X$ be a smooth $n$-dimensional projective variety defined
over an algebraically closed field $k$. Let us fix an ample line
bundle $\O _X(1)$ and a coherent sheaf $F$ on $X$. Let us also fix
a polynomial $\delta\in \QQ[t]$ of degree $\le (n-1)$. When
writing $\delta$ as
$$\delta (m)={\delta _1} \frac{m^{n-1}}{(n-1)!}+{\delta _2} \frac{m^{n-2}}{(n-2)!}+\dots +\delta _n,$$
we will assume that the first non-zero coefficient is positive.

Let us recall a few definitions from \cite{HL1}. A \emph{framed
module} is a pair $(E, \alpha)$, where $E$ is a coherent sheaf and
$\alpha : E\to F$ is a homomorphism. Let us set
$\epsilon(\alpha)=0$ if $\alpha=0$ and $\epsilon(\alpha)=1$ if
$\alpha \ne 0$.  Then we define the \emph{Hilbert polynomial} of
$(E, \alpha)$ as $P(E,\alpha)=P(E)-\epsilon (\alpha) \cdot
\delta$. If $E$ has positive rank then we also define the
\emph{slope} of $(E, \alpha)$ as $\mu (E,\alpha)=(\deg
(E)-\epsilon (\alpha) \delta _1)\cdot \rk E$.

\begin{Definition}
A framed module $(E,\alpha)$ is called \emph{Gieseker
$\delta$-(semi)stable} if for all framed submodules
$(E',\alpha')\subset (E,\alpha)$ we have $\rk E\cdot
P(E',\alpha')(\le) \rk E' \cdot P(E, \alpha).$

If $E$ is torsion free than we say that $(E,\alpha)$ is
\emph{slope $\delta_1$-(semi)stable} if for all framed submodules
$(E',\alpha')\subset (E,\alpha)$ of rank $0<\rk E'<\rk E$ we have
$\mu(E',\alpha')(\le) \mu(E, \alpha).$
\end{Definition}

Let us assume that $F$ is a torsion free sheaf on a divisor
$D\subset X$. In the following we identify $F$ with its push
forward to $X$.

\begin{Lemma} \label{ss-pairs}
Let $E$ be a slope semistable torsion free sheaf on $X$ and let
$E|_D\simeq F$ be a framing. Then the corresponding framed module
$(E, \alpha)$, where $\alpha : E\to E|_D\simeq F$, is slope
$\delta_1$-stable for any small positive constant $\delta_1$.  In
particular, $(E,\alpha)$ is Gieseker $\delta$-stable for all
polynomials $\delta$ of degree $n-1$ with a small positive leading
coefficient.
\end{Lemma}

\begin{proof}
Note that $\ker \alpha =E(-D)$. Let $E'\subset E$ be a subsheaf of
rank $r'<r=\rk E$. If $E'\subset \ker \alpha$ then
$$\mu (E',\alpha')=\mu (E')\le \mu (E)-Dc_1(\O _X(1))^{n-1}<\mu (E)-\delta_1=\mu (E,\alpha ).$$
If $E'\not\subset \ker \alpha$ then
$$\mu (E',\alpha')=\mu (E')-\frac{\delta_1}{r'}\le \mu
(E)-\frac{\delta_1}{r'}<\mu (E)-\frac{\delta_1}{r}=\mu (E,\alpha
),$$ which proves the lemma.
\end{proof}

Now \cite[Theorem 0.1]{HL1}, together with appropriate
modifications in positive characteristic (see \cite{La1} for the
details) imply the following corollary:

\begin{Corollary}
There exists a quasi-projective scheme $M(X; D, F, P)$ which
represents the moduli functor  $\M (X; D, F, P): \Sch/k\to \Sets$,
which to a $k$-scheme of finite type $S$ associates the set of
isomorphism classes of $S$-flat families of pairs $(E, E|_D\simeq
F)$, where $E$ is a slope semistable torsion free sheaf on $X$
with fixed Hilbert polynomial $P$. It can be constructed as an
open subscheme of the projective moduli scheme of Gieseker
$\delta$-stable framed modules $M_{\delta}^{s}(X; D, F,
P)=M_{\delta}^{ss}(X; D, F, P)$ for any polynomial $\delta$ of
degree $n-1$ with a small positive leading coefficient.
\end{Corollary}

Let $X$ be a surface and let $F$ be a semistable locally free
sheaf on a smooth irreducible curve $D\subset X$. Assume that $D$
is numerically proportional to the polarization $c_1(\O_X(1))$.
Then any torsion free sheaf $E$ on $X$ for which there exists a
framing $E|_D\simeq F$ is automatically slope semistable. So in
this case we have a quasi-projective moduli space for torsion free
sheaves with framing without any need to introduce the stability
condition.

This in particular implies that the moduli spaces of torsion free
sheaves $E$ on $\PP ^2$ with fixed rank $r$, second Chern class
and framing $E\simeq \O_{l_{\infty}}^r$ at the fixed line
$l_{\infty}$ can be considered as an open subscheme of the moduli
space of framed modules of \cite{HL1} and it is a fine moduli
space for the corresponding moduli functor (cf. \cite[Remark
2.2]{Na1}).

However, the situation becomes more subtle if we want to consider
moduli spaces of $(r,c)$-instantons $E$ on $\PP ^3$ with framing
$E\simeq \O_{l_{\infty}}^r$ at the fixed line $l_{\infty}\subset
\PP ^3$:

\begin{Proposition}
Let $\ti{E}$ be an $(r-1,c)$-instanton on $\PP ^3$ and let
$E|_{l_{\infty}}\simeq \O_{l_{\infty}}^r$ be a framing of
$E=\ti{E}\oplus \O_{\PP ^3}$. If $c=c_2(E)>r(r-1)$ then $E$ is an
$(r,c)$-instanton but the corresponding framed module $(E,\alpha)$
is not Gieseker $\delta$-semistable for any positive polynomial
$\delta$.
\end{Proposition}

\begin{proof}
Assume $(E,\alpha)$ is Gieseker $\delta$-semistable for some
positive polynomial $\delta$.  Then the stability condition for
$E'=I_{l_{\infty}}E\subset E$ gives $\frac{P(E')}{r}\le
\frac{P(E)-\delta}{r}$, i.e., $\delta \le
P(E)-P(E')=rP(\O_{l_{\infty}})$. Hence $$\delta (m)\le r(m+1)$$
for large $m$. On the other hand, we have $P(\O_{\PP
^3})-\delta\le \frac{P(E)-\delta}{r}$, which translates into
$$\delta (m)\ge \frac{c(m+2)}{r-1}.$$
Hence $c\le r(r-1)$.
\end{proof}

\medskip

Below we show that the moduli space of framed instantons on $\PP
^3$ can be constructed as an open subscheme of the moduli space of
framed modules but on a different variety. Before giving a precise
formulation of this result let us introduce some notation.

Let $\Lambda \simeq \PP^1$ be the pencil of hyperplanes passing
through $l_{\infty}=\{x_0=x_1=0\}$ in $\PP^3$. The coordinates of
this $\PP^1$ are denoted by $y_0, y_1$.  Let $X=\{(H,x): x\in
H\}\subset \PP^1\times \PP^3$ be the incidence variety.  It is
defined by the equation $y_1x_0=y_0x_1$. Let $p$ and $q$ denote
the corresponding projections of $X$ onto $\Lambda$ and $\PP^3$.
We will write $\O_X(a,b)$ for $p^*\O_{\PP^1}(a)\otimes
q^*\O_{\PP^3}(b).$ The projection $q:X\to \PP^3$ is the blow up of
$\PP ^3$ along the line $l_{\infty}$. The exceptional divisor of
$q$ will be denoted by $D$. It is easy to see that $\O_X(D)\simeq
\O_{\PP^3}(-1,1).$ Note that $X$ is equal to the projectivization
of $N=\O^2_{\PP^1}\oplus \O_{\PP^1} (1)$ on $\PP^1$. The relative
$\O_{\PP (N)}(1)$ for this projectivization is equal to
$q^*\O_{\PP^3}(1)$. We will denote this line bundle by $\O_X(1)$.

\begin{Theorem}
There exists a quasi-projective scheme $\M ^f(\PP ^3; r, c)$ which
represents the moduli functor $\ti \M ^f(\PP ^3; r,c): \Sch/k\to
\Sets$ given by
$$S\to\left\{
\begin{array}{c}
\hbox{Isomorphism classes of $S$-flat families } \\
\hbox{of framed $(r,c)$-instantons $E$ on $\PP ^3$.}\\
\end{array}
\right\}$$ It is isomorphic to $M(X;D, \O_D^r, P)$ for a suitably
chosen Hilbert polynomial $P$ and an arbitrary polarization.
\end{Theorem}

\begin{proof}

Let $E$ be an instanton on $\PP^3$.

\begin{Lemma} \label{pull-back}
$q^*E$ is slope $\ti{H}$-semistable for any ample line bundle
$\ti{H}$ on $X$.
\end{Lemma}

\begin{proof}
Let us set $\xi =c_1(\O_X(1)).$ It is easy to see that $q^*E$ is
slope $\xi$-semistable as otherwise the push forward of the
destabilizing subsheaf would destabilize $E=q_*(q^*E)$ (see Lemma
\ref{stab-inst}).

Moreover, the restriction of $q^*E$ to a general fibre of $p$ is
isomorphic to the restriction of $E$ to a hyperplane in $\PP^3$
containing $l$, which is clearly semistable. So $q^*E$ is slope
$f\xi$-semistable (i.e., slope in the semistability condition is
computed as $c_1\cdot f\xi/\rk$).

The nef cone of $X$ is generated by divisors $\xi$ and
$f=p^*c_1(\O_{\PP^1}(1))$.  So we can write $\ti{H}=a\xi+bf$ for
some positive numbers $a$ and $b$.  Then
$\ti{H}^2=a^2\xi^2+2abf\xi $, so slope $\ti{H}$-semistability of
$q^*E$ follows from the above.
\end{proof}

In the proof we also need a generalization of Ishimura's
generalization \cite[Theorem 1]{Is} of Schwarzenberger's theorem.
For a moment let us switch to a different notation:

Let $X$ and $Y\subset X$ be smooth varieties and let $S$ be an
arbitrary noetherian $k$-scheme. Let $\pi: {\ti X}\to X$ be the
blow up of $X$ along $Y$. Let $E$ be the exceptional divisor and
let ${\ti \pi}=\pi|_E:E\to Y$.  Let us set $\pi_S=\pi\times \Id
_S: {\ti X}\times S\to X\times S$ etc.

\begin{Theorem} \emph{(cf. \cite[Theorem 1]{Is})} \label{Ishimura}
Let $\F$ be a coherent sheaf on ${\ti X}\times S$ such that
$\F|_{E\times S}\simeq {\ti \pi}^*\G$ for some locally free sheaf
$\G$ on $Y\times S$. Then the coherent sheaf $\E =\pi_{S*}\F$ is
locally free in an open neighborhood of $Y\times S$ and the
natural map $\pi_S^*\E\to \F$ is an isomorphism.
\end{Theorem}

\begin{proof}
The theorem can be proven in exactly the same way as \cite[Theorem
1]{Is} using the fact that cohomology commutes with flat base
extension.
\end{proof}

Coming back to the proof of the theorem we will show that the
functor $\M (l;r,c)$ is represented by the quasi-projective moduli
scheme $M(X; D, \O_D^r, P)$ (for a suitably chosen $P$ and an
arbitrary fixed polarization).

First let us note that by Lemma \ref{pull-back} there exists a
natural transformation of functors
$$\Phi: \ti \M ^f(\PP ^3;r,c)\to \M (X;D,\O_D^r, P)$$ given by sending a flat
$S$-family $(E_S,E|_{l\times S}\simeq \O _{l\times S}^r)$ of
framed $(r,c)$-instantons to the family $(q_S^*E _S,
q_S^*E_S|_{D\times S}\simeq \O _{D\times S}^r)$. To show the above
claim it is sufficient to prove that the transformation $\Phi$ is
an isomorphism of functors. First note that
$$q_{S*}q_S^*E_{S}\simeq E_S\otimes q_{S*}\O_{X\times S}\simeq E_S,$$
where the first isomorphism comes from the projection formula
(note that $E_S$ is locally free around $l\times S$) and the
second isomorphism follows since push-forward commutes with flat
base extension.  Similarly, we have
$$R^1q_{S*}(q_S^*E_{S}(-D\times S))\simeq E_S\otimes
R^1q_{S*}\O_{X\times S}(-D\times S)=0,$$ so
$q_{S*}(q_S^*E_S|_{D\times S})\simeq E_S|_{l\times S}$ and the
push-forward of $q_S^*E_S|_{D\times S}\simeq \O _{D\times S}^r$
gives an isomorphism $E_S|_{l\times S}\simeq \O _{l\times S}^r$.

Hence Theorem \ref{Ishimura} implies that the  natural
transformation
$$\Psi: \M (X;D,\O_D^r, P)\to \ti \M ^f(\PP ^3;r,c)$$ given by sending a flat
$S$-family $(F_S,F|_{D\times S}\simeq \O _{D\times S}^r)$ to the
family $(q_{S*} F_S, (q_{S*}F_S)|_{l\times S}\simeq \O _{l\times
S}^r)$ is inverse to $\Phi$.
\end{proof}

\section{Deformation theory and smoothness of moduli spaces of instantons}

In this section we give a very quick review of deformation theory
for framed perverse instantons. We sketch only a quite simple fact
from deformation theory used a few times throughout the paper
without going into long technical results showing, e.g., virtual
smoothness of the moduli space of stable perverse instantons.

Then we show that if $E_1$ and $E_2$ are locally free instantons
then $\ext ^2 (E_1,E_2)$ vanishes for low ranks and second Chern
classes. This implies that the moduli space of locally free
instantons embeds as a Lagrangian submanifold into the moduli
space of sheaves on a quartic. It also proves that the moduli
space of framed locally free $(r,c)$-instantons is smooth for low
values of $r$ and $c$.

\subsection{Deformation theory for framed perverse instantons}

\medskip

Let $(\C, \Phi)$ be a stable framed perverse instanton corresponding to
an ADHM datum $x\in \tilde{\bf B}$. Let $\varphi : G\to \BB$ be
the orbit map sending $g$ to $g x$.

\begin{Theorem} \label{deformation}
Let us consider the complex $K$
$$0\to  K^0=\fg \mathop{\longrightarrow}^{d\varphi_e} K^1=T_x \tilde{\bf B} \mathop{\longrightarrow}^{d \ti \mu _x}
K^2 =T_0(\enD (V)\otimes H^0(\O_{\PP ^1}(2)))\to 0$$ Then $H^i
(K)=0$ for $i\ne 1,2$, $H^1(K)={\ext}^1 (\C, J_{l_{\infty}}\otimes
\C)$ and $H^2 (K)={\ext}^2 (\C, \C).$ In particular, if ${\ext}^2
(\C, \C)=0$ then the moduli space $\M (\PP ^3; r,c)$ is smooth of
dimension $4cr$ at $[(\C, \Phi )]$.
\end{Theorem}

\begin{proof}
We know that $\C$ is quasi-isomorphic to the following complex
$$0\to \C^{-1}:=V\otimes \O_{\PP ^3}(-1) \mathop{\to}^{\alpha}
\C^{0}:=\ti{W} \otimes \O_{\PP ^3} \mathop{\to}^{\beta} \C^1:=
V\otimes \O_{\PP ^3}(1)\to 0,$$ where $\dim V=c$ and $\dim
{\ti{W}}=r+2c$ (more precisely $\ti W =V\oplus V\oplus W$) and
$\alpha,\beta$ are defined by the ADHM datum $x$ as in
\ref{FJ-section}. Let us consider the complex $\D=\Hom^{\bullet}
(\C , J_{l_{\infty}}\otimes \C)$. Then we see that
$${\ext}^i(\C, J_{l_{\infty}}\otimes \C)=\HH^i (\PP ^3, \D ),$$ where $\HH^i(X, \D )$ denotes the $i$th
hypercohomology group of the complex $\D$. Let us consider a
spectral sequence
$$H^t(\PP^3, \D ^s)\Rightarrow \HH^{s+t}(\PP ^3, \D).$$
Using this spectral sequence we see that we have a complex
$$0\to L^{0} =H^2(\PP ^3,\D^{-2})\mathop{\to}^{d^0_L} L^{1}=H^0(\PP ^3, \D ^1) \mathop{\to}^{d^1_L} L^{2}=H^0(\PP ^3, \D ^2)\to 0$$
such that $H^1(L)=\HH ^1(\PP ^3, \D)$ and $H^2(L)=\HH ^2(\PP ^3,
\D)$. We have $L^0=\hom (V,V)\otimes H^2(\PP ^3,
J_{l_{\infty}}(-2))$, $L^1=(\hom (\ti W, V)\oplus \hom (V, \ti
W))\otimes H^0(\PP ^3, J_{l_{\infty}}(1))$ and $L^2 =\hom
(V,V)\otimes H^0(\PP ^3, J_{l_{\infty}}(2))$. Note that $H^2(\PP
^3, J_{l_{\infty}}(-2))\simeq k$ but $H^0(\PP ^3,
J_{l_{\infty}}(2))\simeq k^7$, so this is not yet the complex we
were looking for. However, if we write down everything in
coordinates we see that $d^1_L$ is an isomorphism on $\hom (V,V)
\otimes k^4$ and after splitting off the corresponding factors
from $L^1$ and $L^2$ we get exactly complex $K$. Obviously, we
need to write down everything in coordinates to check that the
obtained maps are essentially the same. We leave the details to
the reader. Now the theorem follows from the following lemma:

\begin{Lemma} \label{porownanie}
Let $\C$ be a framed perverse $(r, c)$-instanton on $\PP^3$. Then
$${\ext}^2(\C, \C)={\ext}^2(\C,J_{l_{\infty}}\otimes \C).$$
\end{Lemma}

\begin{proof}
We have a distinguished triangle
$$J_{l_{\infty}} \otimes \C \to \C \to \C\otimes \O_{l_{\infty}}\to J_{l_{\infty}} \otimes \C [1].$$
This triangle gives
$${\ext}^1(\C, \C\otimes \O_{l_{\infty}})\to {\ext}^2(\C, J_{l_{\infty}}\otimes \C)\to
{\ext}^2(\C, \C)\to {\ext}^2(\C, \C\otimes \O_{l_{\infty}}). $$
But ${\ext}^l(\C, \C\otimes \O
_{l_{\infty}})=h^l(\O_{l_{\infty}}^{r^2})=0$ for $l=1,2$, so we
get the required equality.
\end{proof}

This finishes proof of Theorem \ref{deformation}.
\end{proof}

\medskip

\begin{Remark}
Let $(\C, \Phi)$ be a stable framed perverse instanton. Then by a
standard computation  one can see that  the tangent space to $\M
(\PP ^3; r,c)$ at the point corresponding to $(\C, \Phi)$ is
isomorphic to ${\ext}^1 (\C, J_{l_{\infty}}\otimes \C)$. Moreover,
one can show that there exists an appropriate obstruction theory
with values in ${\ext}^2 (\C, \C)$ (cf. \cite[2.A.5]{HL2}).
\end{Remark}

\subsection{Smoothness of the moduli space of framed locally free
instantons}

\begin{Lemma}\label{rest-plane}
Let $E$ be a locally free instanton of rank $r=2$ or $r=3$. Then
for any plane $\Pi\subset \PP ^3$ the restriction $E_{\Pi}$ is
slope semistable.
\end{Lemma}

\begin{proof}
Let us note that we have a long exact cohomology sequences:
$$0=H^0(E(-1))\to H^0(E_{\Pi}(-1))\to H^1(E(-2))=0$$
and
$$0=H^0(E^*(-1))\to H^0(E^*_{\Pi}(-1))\to H^1(E^*(-2))\simeq (H^2(E(-2)))^*=0,$$
where the isomorphism in the second sequence comes from the Serre
duality. This implies that $E_{\Pi}(-1)$ and $E^*_{\Pi} (-1)$ have
no sections which  in ranks $2$ and $3$ implies semistability of
$E_{\Pi}$.
\end{proof}

\medskip

\begin{Lemma} \label{bound}
Let $E_i$  be a locally free $(r_i,c_i)$-instanton on $\PP ^3$,
where $i=1,2$. Then $\ext ^2(E_1, E_2 (-2))$ has dimension at most
$c_1c_2$.
\end{Lemma}

\begin{proof}
Our assumption implies that $E_i$ is the cohomology of the
following monad $\C_i^{\bullet}$
$$0\to V_i\otimes \O_{\PP ^3}(-1) \mathop{\to}^{d^{-1}_{\C_i}}
{\ti{W}}_i\otimes \O_{\PP ^3} \mathop{\to}^{d^0_{\C_i}} V_i\otimes
\O_{\PP ^3}(1)\to 0,$$ where $\dim V_i=c_i$ and $\dim
{\ti{W_i}}=2c_i+r_i$. Let us consider the complex
$\C^{\bullet}=\Hom^{\bullet} (\C_1^{\bullet}, \C_2^{\bullet})$
defined by
$$\C ^i:= \bigoplus_k \Hom (\C_1^{k}, \C_2^{k+i})$$ with
$d(f):=d_{\C_2^{\bullet}}\circ f-(-1)^{\deg f}f\circ
d_{\C_1^{\bullet}}.$ Since $\C_i^{\bullet}$ are complexes of
locally free sheaves we see that
$${\ext}^p(E_1, E_2(-2))=\HH^p (\PP ^3, \C^{\bullet}\otimes \O_{\PP ^3}(-2)),$$ where $\HH^p$ denotes the $p$th
hypercohomology group. But then the spectral sequence
$$ H^t (\PP ^3, \C^{s}\otimes \O_{\PP ^3}(-2))\Rightarrow \HH^{s+t} (\PP ^3, \C^{\bullet}\otimes \O_{\PP ^3}(-2))$$
gives an exact sequence
$$ 0\to \ext ^1(E_1, E_2 (-2))\to \hom (V_1,V_2)\to \hom (V_1,V_2)
\to \ext ^2(E_1, E_2 (-2))\to 0.
$$
Clearly, this implies the required inequality.
\end{proof}

\medskip
The proof of the following theorem uses the method of proof of
\cite[Th\'eor\`eme 1]{LP}.

\begin{Theorem} \label{ext-vanishing}
 Let $E_i$  be a locally free $(r_i,c_i)$-instanton on
$\PP ^3$, where $i=1,2$. If $r_1, r_2\le 3$ and $c_1c_2\le 6$ then
$\ext ^2(E_1, E_2 )=0$.
\end{Theorem}

\begin{proof}
Let $Z=\{(x,\Pi)\in \PP^3\times (\PP ^3)^* : x\in \Pi\}$ be the
incidence variety of planes containing a point in $\PP ^3$. Let
$p_1,p_2$ denote projections from $\PP^3\times (\PP ^3)^*$ onto
$\PP^3$ and $ (\PP ^3)^*$, respectively, and let us set
$q_1=p_1|_Z$ and $q_2=p_2|_Z$. On $ \PP^3\times (\PP ^3)^*$ we
have a short exact sequence
$$0\to \O_{ \PP^3\times (\PP ^3)^*}(-1,-1)\to \O_{ \PP^3\times (\PP ^3)^*}\to \O_{Z}\to 0.$$
Let us tensor this sequence with $p_1^*\Hom (E_1,E_2(i))$ and push
it down by $p_2$. Then we get an exact sequence
$$
\ext^2(E_1,E_2(i-1))\otimes \O_{(\PP ^3)^*}(-1)
\mathop{\to}^{\varphi_i} \ext^2(E_1,E_2(i))\otimes \O_{(\PP ^3)^*}
\to R^2q_{2*}q_1^*\Hom (E_1,E_2(i)).
$$
But for any plane $\Pi\subset \PP ^3$ the group  $\ext
^2((E_1)_{\Pi},(E_2)_{\Pi}(i))$ is Serre dual to $\hom
((E_2)_{\Pi},(E_1)_{\Pi}(-i-3)).$ By Lemma \ref{rest-plane} both
$(E_1)_{\Pi}$ and $(E_2)_{\Pi}$ are semistable of the same slope
so if $i>-3$ then $\hom ((E_2)_{\Pi},(E_1)_{\Pi}(-i-3))=0$. This
implies that $R^2q_{2*}q_1^*\Hom (E_1,E_2(i))=0$ for $i>-3$ and
hence for such $i$ we have a short exact sequence
$$0\to F_i=\ker \varphi _i\to \ext^2(E_1,E_2(i-1))\otimes \O_{(\PP ^3)^*}(-1)
\mathop{\to}^{\varphi_i} \ext^2(E_1,E_2(i))\otimes \O_{(\PP
^3)^*}\to 0.$$ Now $F_i$ is a vector bundle (again only for
$i>-3$). Let $s_i$ denotes its rank. If $s_i <3$ and
$\ext^2(E_1,E_2(i))\ne 0$ then $c_{s_i+1}(F_i)$ is non-zero which
contradicts the fact that $F_i$ is locally free. Therefore if
$\ext^2(E_1,E_2(i))\ne 0$ for some $i\ge -2$ then
$$s_i=
\dim \ext^2(E_1,E_2(i-1))-\dim \ext^2(E_1,E_2(i))\ge 3.$$ Applying
this inequality for $i=0$ and $i=-1$ we see that if $\ext ^2(E_1,
E_2 )\ne 0$ then $\ext ^2(E_1, E_2 (-2))$ has dimension at least
$7$. By Lemma \ref{bound} this contradicts our assumption on
$c_1c_2$.
\end{proof}

\begin{Corollary}\label{smoothness-for-low-charge}
Let $r\le 3$ and $c\le 2$. Then the moduli space of framed locally free
$(r,c)$-instantons is smooth of dimension $4cr$.
\end{Corollary}

\begin{proof}
Let $E$ be a locally free $(r,c)$-instanton. Then $\ext ^2(E,
E)=0$ and by Theorem \ref{deformation} the tangent space to the moduli space
is isomorphic to $\ext ^1 (E, J_{l_{\infty}}E)$, so the dimension is equal to
$4cr$.
\end{proof}

\bigskip
Let $M_{\PP ^3}(r,c)$ denotes the moduli space of Gieseker stable
locally free $(r,c)$-instantons on $\PP ^3$. In case of rank $r=2$
or $3$ Gieseker stability of instanton $E$ is equivalent to
$h^0(E)=h^0(E^*)=0$.

Let $S\subset \PP ^3$ be any smooth quartic with $\pic S=\ZZ$.
Let $M_S(r,4c)$ denotes the moduli space of
slope stable vector bundles on $S$ with rank $r$ and Chern classes
$c_1=0$ and $c_2=c \cdot h^2|_S=4c$ ($h$ stands for the class of a
hyperplane in $\PP ^3$).

Using an idea of A. Tyurin (see \cite[Section 9]{Be}) one can show
the following theorem:

\begin{Theorem} \label{lagrangian}
Let $r\le 3$ and $c\le 2$. Then $M_{\PP ^3}(r,c)$ is smooth and
the restriction $r: M_{\PP ^3}(r,c) \to M_S(r,4c)$ is a morphism
which induces an isomorphism of $M_{\PP ^3}(r,c)$ onto a
Lagrangian submanifold of $M_S(r,4c)$.
\end{Theorem}

\begin{proof}
Smoothness of $M_{\PP ^3}(r,c)$ follows directly from Theorem
\ref{ext-vanishing}. To prove that the restriction map $r: M_{\PP
^3}(r,c) \to M_S(r,c)$ is a morphism we need the following lemma:

\begin{Lemma}\label{rest-quartic}
Let $E$ be a locally free instanton of rank $r=2$ or $r=3$. Assume
that $h^0(E)=h^0(E^*)=0$. Then $E$ is slope stable and for any
smooth quartic $S\subset \PP ^3$ with $\pic S=\ZZ$, the
restriction $E_{S}$ is slope stable.
\end{Lemma}

\begin{proof}
The (saturated) destabilizing subsheaf of $E$ has either rank $1$
and then it gives a section of $E$ or it has rank $2$ and then $E$
has rank $3$ and the determinant of the destabilizing subsheaf
gives a section of $\wedge ^2E\simeq E^*$. This proves the first
assertion. By the same argument to show the second assertion it is
sufficient to prove that $h^0(E_S)=h^0(E^*_S)=0$. But this follows
from sequences:
$$0=H^0(E)\to H^0(E_S)\to H^1(E(-4))=0$$
and
$$0=H^0(E^*)\to H^0(E^*_S)\to H^1(E^*(-4))\simeq (H^2(E))^*=0.$$
\end{proof}

\medskip

Let $E$ be a Gieseker stable locally free $(r,c)$-instanton. By
Lemma \ref{rest-quartic} and Theorem \ref{ext-vanishing} we know
that the restriction of $E$ to $S$ is slope stable and $\ext
^2(E,E)=0$. This implies that $r$ is an immersion at the point
$[E]$ (see \cite[9.1]{Be}). Therefore we only need to show that
$r$ is an injection.

To prove that let us take two Gieseker stable locally free
$(r,c)$-instantons $E_1$ and $E_2$. Then by Theorem
\ref{ext-vanishing} we have an exact sequence
$$\hom (E_1,E_2)\to \hom ((E_1)_S,(E_2)_S)\to \ext ^1(E_1,E_2(-S))\simeq (\ext ^2(E_2,E_1))^*=0.$$
This shows that we can lift any isomorphism $(E_1)_S\to (E_2)_S$
to an isomorphism of $E_1$ and $E_2$ and hence $r$ is injective.
\end{proof}

\medskip

\begin{Remark}
It is very tempting to conjecture that Theorem \ref{ext-vanishing}
holds for all pairs of locally free instantons (maybe with some
additional assumptions concerning stability of these bundles).
This would imply a well known conjecture on smoothness of the
moduli space of locally free instantons. Even then, an analogue of
Theorem \ref{lagrangian} does not immediately follow. But if one
restricted to the open subset of bundles for which all exterior
powers remain instantons then one could embed it into $M_S(r,4c)$
as a Lagrangian submanifold.

However, it seems that all these conjectures are just a wishful
thinking similar to the original conjecture on smoothness of the
moduli space of locally free instantons: there are very few known
results and all the methods work only for instantons of low
charge.
\end{Remark}

\medskip

\begin{Example}
For $r=2$ and $c=1$ the moduli space $M_{\PP ^3}(2,1)$
parameterizes only null-correlation bundles and it is known that
$M_{\PP ^3}(2,1)\simeq \PP ^5\setminus \Gr (2,4)$, where $\Gr
(2,4)$ is the Grassmannian of planes in $\AA ^4$ (see
\cite[Chapter II, Theorem 4.3.4]{OSS}). By the above theorem this
is a Lagrangian submanifold of the moduli space $M_S(2,4)$. Over
complex numbers $M_S(2,4)$ is known to have a smooth
compactification to a holomorphic symplectic variety (see
\cite{OG}). Note that Lagrangian fibrations $M_S(2,4)\to \PP ^5$
for some K3 surfaces $S$ were constructed by Beauville in
\cite[Proposition 9.4]{Be}. It is possible that the Lagrangian
submanifold $M_{\PP ^3}(2,1)$ extends to a section of some
Lagrangian fibration (possibly after deforming the
compactification) providing another example when this is possible
(see \cite{Sa} for the proof that some Lagrangian fibrations can
be deformed to Lagrangian fibrations with a section in case of
4-dimensional varieties).
\end{Example}

\section*{Acknowledgements}

A large part of this paper contains results from the un-submitted
PhD thesis of the first author written under supervision of the
second author. Some of the other results of this paper were
obtained by the second author while his stay at the Max Planck
Institut f\"ur Mathematik in Bonn in 2007. The rather misleading
report \cite{La2} which reported on work carried out in Bonn
contains some false counter-examples, as well as statements whose
proofs turned out to be much more complicated than expected. We
believe that the present article rights all of these wrongs.

The second author was partially supported by a grant of Polish
Ministry of Science and Higher Education (MNiSW).

\end{document}